\documentclass{article}

\usepackage{amsmath}
\usepackage{amsthm} 
\usepackage{amsfonts}
\usepackage{graphicx}
\usepackage{epstopdf}
\usepackage{algorithmic} 
\usepackage{algorithm} 

\newtheorem{theorem}{Theorem}[section]
\newtheorem{lemma}[theorem]{Lemma}
\newtheorem{corollary}[theorem]{Corollary}

\newtheorem{conjecture}{Conjecture}
\newtheorem{assumption}{Assumption}
\newtheorem{remark}{Remark}

\newcommand{\secref}[1]{\S \ref{#1}}
\newcommand{\algref}[1]{Algorithm \ref{#1}}
\newcommand{\figref}[1]{Figure \ref{#1}}
\newcommand{\asref}[1]{Assumption \ref{#1}}

\title{A positivity preserving convergent event based asynchronous PDE solver}
\author{Daniel Stone \and Gabriel Lord}

\begin{document}
\maketitle
\begin{abstract}
A new numerical scheme for conservation equations based on evolution by asynchronous discrete events is presented. During each event of the scheme only two cells of the underlying Cartesian grid are active, and an event is processed as the exact evolution of this subsystem. This naturally leads to and adaptive scheme in space and time. Numerical results are presented which show that the error of the asynchronous scheme decreases to zero as a control parameter is reduced. The construction of the scheme allows it to be expressed as repeated multiplications of matrix exponentials on an initial state vector; thus techniques such as the Goldberg series and the Baker--Campbell--Hausdorff (BCH) formula can be used to explore the theoretical properties of the scheme. We present the framework of a convergence proof in this manner.
\end{abstract}

\section{Introduction}
We develop and present analysis of new schemes for the simulation of
porous media flow based on an asynchronous simulation methodology;
that is, one in which different parts of the spatial domain ar
allowed to exist at different times simultaneously during the course
of the simulation. Specifically the schemes are example of the
Discrete Event Simulation (DES) methodology, in which the system is
evolved forward in time by discrete events, local in space, with each
event having its own timestep determined by the local physical
activity in the region \cite{OK_plasma, OK_flux,
  async_gas_discharge}. 
In our case we take this to be the magnitude of the local flux. The schemes we
develop are intended to be applied to advection-diffusion type
conservation law systems,
\begin{equation}
\frac{dc(\mathbf{x},t)}{dt} = \nabla f(c(\mathbf{x},t)),  \qquad t \in \mathbb{R}, \qquad \mathbf{x} \in \Omega \subset \mathbb{R}^d,
\label{con_no_source}
\end{equation} 
$d= 1,2,3$, where $c(\mathbf{x},t)$ is a concentration and $f$ is a
given flux function. An initial condition $c(\mathbf{x},0) =
c_0(\mathbf{x})$ is provided. We
consider `no flow' boundary conditions, that is, Neumann type boundary
conditions with zero flux on external faces, however other types of boundary
conditions can easily be added in this framework without much
difficulty. 

The idea behind these schemes is essentially simple, but
unusual. Consider the spatial domain $\Omega$ discretised in a way
standard in the finite volume methodology, with approximate fluxes
defined on every face. The evolution of the system forward in time
proceeds by means of discrete local events: the transfer of mass
across a single face, between the two adjacent cells. Which faces are
given priority for events, and the timescales of the events, are to be
functions of the flux across the face - in general, a flux of greater
magnitude has the effect of sooner and shorter events. 

Although initially developed for discrete systems, DES has been applied in 
\cite{OK_plasma, OK_flux,  async_gas_discharge} for plasma simulation,
one-dimensional conservation laws and gas discharge with high levels
of accuracy and efficiency. The methods in these papers are cell based.
\cite{Async1,mythesis} introduced the Basic Asynchronous
Scheme (BAS), a face based method that we use to  
compare to an improved method - the Exact Asynchronous Scheme (EAS)
that we consider here. EAS is improved in that it is automatically positivity preserving
and more accurate, as will be seen, for example in \secref{numerical} and \secref{pos-sec}.

The faced based DES approaches BAS and EAS have some similarities
with the approximate Riemann solvers of Roe, see
\cite{roe1981approximate} and, for example, \cite{toro2009riemann}. In
a Roe-type solver, the spatial discretisation is viewed as producing a
series of Riemann problems (i.e., a conservation equation with
discontinuous initial data), one at each face in the grid. Each
Riemann problem can be approximately solved by introducing a matrix
approximation at the face with certain properties. 

The paper is arranged as follows. In \secref{general_form} we describe the general form of the schemes and give an overview of the finite volume discretisation upon which they are based. We also introduce a way of expressing the full finite volume system as an accumulation of simpler subsystems; that is, of expressing the finite volume discretisation matrix $L$ in as an accumulation of `connection matrices' which represent the discretisation of two-cell artificial subsystems. These are the subsystems on which discrete events take place, and the corresponding connection matrices will be crucial in the analysis in \secref{an_sec_as_1}. \\
In \secref{general_form} we present the new schemes in detail. In \secref{numerical} we present numerical results. Discussion of the observed properties and some steps towards analysis are presented in \secref{an_sec_as_1}.


\section{The General Form of the Schemes}
\label{general_form}

Consider a spatial domain $\Omega$ discretised into a cartesian grid
of cells, each with 
a unique index $j \in \{ 1,2,\ldots, J \} = \mathcal{C}$.
Similarly let every face also have a unique index $k \in  \{
1,2, \ldots K \} = \mathcal{F}$. 
For a cell with index $j \in
\mathcal{C}$, define the set $\mathcal{F}_j$ of faces belonging to the
cell, where $\mathcal{F}_j \subset \mathcal{F}$. Also, define the set
of associated faces $\mathcal{\tilde{F}}_k$ of a face $k$ as
follows. If face $k$ is adjacent to cells $j_1, j_2 \in \mathcal{C}$,
then  
$$
\mathcal{\tilde{F}}_k = \mathcal{F}_{j_1}\cup \mathcal{F}_{j_2},
$$
i.e., the associated faces is the set of all the faces of the two
cells which face $k$ is adjacent to. 

The algorithms work on the principle that each face $k$ in the grid
has an independent time and update time value. The update time is
connected to the current time, the flux across the face, and a global
control parameter which we call the mass unit $\Delta M$. 

The evolution of the system proceeds by a sequence of discrete events in which an amount of mass $\delta M < \Delta M$ is transferred across a single face $k$. Which face is chosen for an event is described in \secref{algsec}. We will refer to the face currently undergoing an event as the active face.

Let $f_k$ be
the approximation of the flux on a face $k \in \mathcal{F}$, which
depends upon the concentration values $c_{j_1}$, $c_{j_2}$ in the two
cells  with indexes $j_1,j_2 \in \mathcal{C}$ adjacent to face
$k$. The concentration $c_j$ of a cell $j$ is assumed constant
throughout the cell, and is derived from the mass in the cell $m_j$
and its volume $V_j$ as $c_j = \frac{m_j}{V_j}$. The flux $f_k$ on a
face is assumed constant and defines the flow of mass across the face
between its two adjacent cells, i.e., the flow of mass from cell $j_1$
due to face $k$ will be $-f_k A_k$; and into cell $j_2$ will be be
$f_k A_k$, where $A_k$ is area of the face $k$. The direction of mass
flow depends on the sign on $f_k$. To be explicit, the equations for
mass flow \emph{across a single face} $k$, are 
\begin{equation}
\frac{dm_{j_1}}{dt} = f_k A_k, \qquad \frac{dm_{j_2}}{dt} = -f_k A_k. 
\label{mass_face_flux}
\end{equation}
Let $\bar{D}_k$ be an approximation of the diffusivity at the face
based on the diffusivity in the two cells (typically the harmonic mean
of $D_{j_1}$ and $D_{j_2}$) and let $\Delta x_k$ be the distance
between the two cell centroids. 
For the advection-diffusion system \eqref{con_no_source}, one of the
two cells will be the upwind cell; without loss of generality let this
be cell $j_1$. Then the flux may be approximated by finite differences
\begin{equation}
  f_k = \frac{\bar{D}_k
    \left(\frac{m_{j_2}}{V_{j_2}}-\frac{m_{j_1}}{V_{j_1}}
    \right)}{\Delta x_k} - \frac{m_{j_1}}{V_{j_1}} v , 
  \label{fd_ad_flux}
\end{equation}
where $v$ is the scalar product of the velocity at the centre of face
$k$ with the unit vector in the direction of the line from the centre
of cell $j_1$ to cell $j_2$. 

The total rate of change of mass, and thus concentration in a cell $j$
is the sum of \eqref{mass_face_flux} for each $k \in
\mathcal{F}_{j}$. This can be expressed as a matrix, $L$ which gives
the finite volume semidiscretisation of (\ref{con_no_source}) as a
system of ODEs,  
\begin{equation}
\frac{d  \mathbf{c}}{dt}=L  \mathbf{c},  \qquad L \in \mathbb{R}^{J \times J}
\label{fv_disc}
\end{equation}
where $\mathbf{c} = (c_1, c_2, \ldots , c_J)^T$ is the vector of
concentrations in cells. In a standard finite volume based
implementation \eqref{fv_disc} is then discretised in time, resulting
in the fully discrete approximation. In contrast, face based
asynchronous schemes do not form the global system (\ref{fv_disc}) but
are based on events involving the transfer of mass across a single
face.

Consider a single face $k$ in the finite volume grid. Let the two
cells that this face neighbours be referred to as $j_1$ and $j_2$. 
Now, consider a sparse matrix $L_k$ with nonzero elements at
$(j_1,j_1)$, $(j_1,j_2)$, $(j_2,j_1)$ and $(j_2,j_2)$; 
\begin{equation}
L_k \equiv
\left(
\begin{array}{ccccc}
 & & & & \\
 &-a_k & &b_k & \\
 & & & & \\
  &a_k & &-b_k & \\
 & & & & \\
\end{array}
\right)
\in \mathbb{R}^{J \times J}.
\label{sparse_c_mat}
\end{equation}
Let $\mathbf{m}$ be the vector of all mass values in the system and
$\mathbf{c}$ the vector of all concentration values in the system,
related by $\mathbf{c}=\mathbf{m} \mathbf{V}$, where $\mathbf{V}$ is
the diagonal matrix with entries $\frac{1}{V_j}$, i.e., the inverse of
the volume in each cell. 
The connection matrix $L_k$ for the face $k$ is such
that the finite volume discretisation matrix $L$ in (\ref{fv_disc})
can be accumulated from the connection matrices on each face, that is, 
\begin{equation}
L = \sum_{k \in \mathcal{F}}{L_k}.
\label{L_as_sum}
\end{equation}

This is achieved if the  scalars $a_k,b_k$ are defined as functions of
the diffusivity $D_j$ and velocity $v_j$ of the two cells, the
distance between their centres, and the area of the face $k$, as
follows. Recalling equations \eqref{mass_face_flux}  
and \eqref{fd_ad_flux}, we see that, if $j_1$ is the upwind cell,
we should set $a$ and $b$ to, 
$$
a_k = \bar{D}_k \frac{1}{V_{j_1} \Delta x_k} + v, \qquad b_k = \bar{D}_k \frac{1}{V_{j_2} \Delta x_k},
$$
or, if $j_2$ is the upwind cell,
$$
a_k = \bar{D}_k \frac{1}{V_{j_1} \Delta x_k} , \qquad b_k = \bar{D}_k \frac{1}{V_{j_2} \Delta x_k} + v,
$$
where $v$ is the scalar product of the velocity at the centre of the
face, with the unit vector in the direction of the line connecting the
centres of the two cells, pointing from the upwind into the downwind
cell. Note that this ensures that $a_k$ and $b_k$ are both
non-negative, since $\bar{D}_k$ and $v$ are both non-negative.

The structure of the connection matrix reflects the conservation of
mass between the two adjacent cells (since the column sum is
zero). The connection matrix $L_k$ associated with face $k$ describes
the relationship between the two cells $j_1$ and $j_2$ adjacent to
face $k$ in the discretisation \eqref{fv_disc}, and thus has nonzero
entries only in columns and rows $j_1$ and $j_2$. Since our
asynchronous schemes operate on a single face (and thus pair
of cells) at a time, connection matrices will prove to be valuable for
re-expressing and analysing our schemes. 

It can be shown that for each $L_k$ there exists a $\hat{\mathbf{z}}_k \in \mathbb{R}^{J}$ such that, for any $\mathbf{x}  \in \mathbb{R}^{J}$,
$$
L_k \mathbf{x} = (b_k x_{j_2} - a_k x_{j_1}) \hat{\mathbf{z}}_k,
$$
where $\hat{\mathbf{z}}_k$ has nonzero entries only at its $j_1$ and $j_2$ positions, and one of these entries is $1$ and the other $-1$. This will be useful later.

Following from the outline at the start of this section,
during a single event only the two cells adjacent to the active face
$k$ is updated. We temporarialy consider the entire system as only
having activity across the face $k$. That is, instead of the system  
$$
\frac{d \mathbf{m}(t)}{dt}  = L  \mathbf{m}(t),
$$
we consider the much simpler system
\begin{equation}
\frac{d \mathbf{m}(t)}{dt}  = L_k  \mathbf{m}(t).
\label{LkSystem}
\end{equation}
It is the reduction of the whole system to the artificially isolated
subsystem \eqref{LkSystem} for discrete events that characterises the
new schemes; and it is the method of solving or approximating
\eqref{LkSystem} which distinguishes them.

\subsection{BAS - Basic Asynchronous Scheme}
\label{bassec}
The simplest face based asynchronous scheme we call BAS, or the Basic Asynchronous scheme, introduced in \cite{Async1}. As stated above, in each event an amount of mass $\delta M \leq \Delta M$ is passed across the active face $k$. In this scheme, we simply ensure that $\delta M = \Delta M$. We will now describe the relations between the time, update time, flux and $\Delta M$.  Let the current individual time of the face $k$ be $t_k$, and its current calculated update time be $\hat{t}_k$. Let $\Delta t = \hat{t}_k  - t_k$ be the timestep associated with face $k$'s next event.

A standard Euler type, approximate solution to \eqref{LkSystem} is,
\begin{equation}
\mathbf{m}(t+\Delta t) \approx  \mathbf{m}(t) + \Delta t L_k  \mathbf{m}(t),
\label{LkSolEuler}
\end{equation}
for the time interval $\Delta t$. In equation \eqref{LkSolEuler} only two cells, the ones adjacent to $k$, are updated, due to the sparsity of $L_k$. Without loss of generality, let the mass in the first cell be $m_{j_1}(t)$ and the second be $m_{j_2}(t)$, and let the first nonzero entry in $\hat{\mathbf{z}}_k$ be $+1$ and the second be $-1$. The evolution of specific cells from the approximation \eqref{LkSolEuler} is 
\begin{equation}
\begin{split}
&m_{j_1}(t+\Delta t) = m_{j_1}(t) +  \Delta t |f_k(t)| A_k \\
& m_{j_2}(t+\Delta t) = m_{j_2}(t) -  \Delta t |f_k(t)| A_k .
\label{BAS_indiv}
\end{split}
\end{equation}
This is equivalent to transferring a mass $ \delta M =  \Delta t |f_k(t)| A_k $ between the two cells. We ensure that $\delta M = \Delta M$ by controling $\Delta t$. We calculate this as 
\begin{equation}
\Delta t  = 
\begin{cases}
& \frac{\Delta M }{|f_k| A_k} \mbox{  if $t_k + \frac{\Delta M }{|f_k| A_k}   \leq T$} \\
& T - t_k \mbox{  otherwise,}
\end{cases}
\label{utime2}
\end{equation}
where the second case is to ensure that the faces synchronise to the desired final time $T$ at the end of the simulation. The first case in \eqref{utime2} is from an Euler-type approximation of \eqref{mass_face_flux};
\begin{equation}
\frac{\delta M}{\Delta t} \approx |\frac{dm}{dt}| = |f_k|A_k; \quad \Delta t \approx \frac{\delta M }{|f_k| A_k} ,
\label{approx-time}
\end{equation}
which provides an approximation of the timescale $\Delta t$ in which $\delta M$ of mass would pass through the face, assuming $\delta M$ small and $f_k$ constant. To obtain \eqref{utime2} we simply set $\delta M = \Delta M$ in \eqref{approx-time}.

Combining \eqref{utime2} with \eqref{BAS_indiv} we see that the first case in \eqref{utime2} indeed gives $\delta M = \Delta t |f_k(t)| A_k  = \Delta M$. When the second  case in \eqref{utime2} is used, an appropriate value $\delta M < \Delta M$ is used instead, again making use of the approximation \eqref{approx-time}: since $\Delta t$ is fixed in this case, we simply approximate $\delta m = \Delta t |f_k|  A_k  $. The general rule for BAS is,
\begin{equation}
\delta m = 
\begin{cases}
&\Delta M \mbox{  if $t^k + \frac{\Delta M }{|f_k| A_k} \leq T$} \\
&  |f_k| (T- t_k) A_k \mbox{  otherwise.}
\end{cases}
\label{del_m}
\end{equation}
Numerical solutions and preliminary analysis of BAS may be found in \cite{Async1}.

\subsection{EAS}
The improved sheme presented here is the the exact-mass asynchronous scheme, or EAS. It differs from BAS in that for the mass transfered is based on an exact solution to the artificially isolated subsystem \eqref{LkSystem}. The timestep update rules are the same as for BAS - that is, \eqref{utime2} is used in EAS.  We now show the form of the exact solution to \eqref{LkSystem}.

\subsubsection{Exact Solution of the Artificially Isolated Subsystem}

A connection matrix acting on any vector produces a vector pointing in only one direction in the solution space. That is, the action of a connection matrix $L_k$ on any vector $\mathbf{x}$ is a scalar multiple of a vector $\hat{\mathbf{z}}_k$, determined by $L_k$. Consider a connection matrix $L_k$ with non-empty columns and rows $j_1$, $j_2$, then
\begin{equation}
L_k \mathbf{x} =
(b_k x_{j_2} -a_k x_{j_1}) \hat{\mathbf{z}}_k,
\label{con_fix_dir}
\end{equation} 
where $\hat{\mathbf{z}}_k = (0, \ldots, 0, 1, 0, \ldots, 0, -1, 0,\dots, 0)^T $, where the non-zero entries are at $j_1$ and $j_2$. It follows that $\hat{\mathbf{z}}_k$ is an eigenvector of $L_k$ and the corresponding eigenvalue can be found,
\begin{equation}
L_k \hat{\mathbf{z}}_k = \lambda_k \hat{\mathbf{z}}_k \qquad \lambda_k = -(a_k+b_k),
\label{evalue_con_mat}
\end{equation}
thus the eigenvalue $\lambda_k$ is negative. From this we can prove the following Lemma.
\begin{lemma}
Let $L_k$ be a connection matrix corresponding to face $k$, which has adjacent cells $j_1$, $j_2$. Let $\hat{\mathbf{z}}_k$ be the eigenvector of $L_k$ be with eigenvalue $\lambda_k$, according to \eqref{evalue_con_mat}. 
Then, for any scalar $s$ and vector $\mathbf{x}$,
\begin{equation}
(e^{sL_k}-I)\mathbf{x} = s(b_k x_{j_2} -a_k x_{j_1})  \varphi_1(-s(a_k+b_k)) \hat{\mathbf{z}}_k,
\label{phi_con_mats_eq}
\end{equation}
where the $\varphi-$function $\varphi_1(\cdot)$ is defined (see for example \cite{NW}) as
\begin{equation}
\varphi_1(z) = z^{-1}\left( e^{z} - I \right) = \sum_{i=0}^{\infty}{\frac{z^i}{(i+1)!}}.
\label{phi1-1}
\end{equation}
\label{phi_con_mats}
\end{lemma}

\begin{proof}
We have that
$$
(e^{sL_k} - I)\mathbf{x} = \sum_{i=1}^{\infty}{\frac{(sL_k)^i}{i!} }\mathbf{x} =  s(b_k x_{j_2} -a_k x_{j_1})  \sum_{i=1}^{\infty}{\frac{(s L_k)^{i-1}}{i!} } \hat{\mathbf{z}}_k.
$$
Since $\hat{\mathbf{z}}_k$ is an eigenvalue of $L_k$, the sum becomes a scalar sum of powers of the eigenvalue $\lambda_k$, and we can shift the index to get
$$
(e^{sL_k} - I)\mathbf{x} =  s(b_k x_{j_2} -a_k x_{j_1})  \sum_{i=0}^{\infty}{\frac{(s\lambda_k)^{i}}{(i+1)!} } \hat{\mathbf{z}}_k.
$$
The series is $\varphi_1(s \lambda_k)$, by \eqref{phi1-1}, and using \eqref{evalue_con_mat} we have \eqref{phi_con_mats_eq}.
\end{proof}

Define the parameter $ \Delta \hat{M}_{k, j_1, j_2} (s)$ to be,
\begin{equation}
 \Delta \hat{M}_{k, j_1, j_2}= s(b_k x_{j_2} -a_k x_{j_1})  \varphi_1(-s(a_k+b_k)),
\label{dme}
\end{equation}

As a consequence of Lemma \ref{phi_con_mats} we can write
\begin{equation}
e^{sL_k}\mathbf{x} = \mathbf{x} +  \Delta \hat{M}_{k, j_1, j_2} \hat{\mathbf{z}}_k(s),
\label{phi_con_2}
\end{equation}

The exact
solution to \eqref{LkSystem} can be re-expressed using
\eqref{phi_con_2}; 
\begin{equation}
 \mathbf{m}(t+\Delta t) =  e^{ t L_k}\mathbf{m}(t) = \mathbf{m}(t)+
 \Delta \hat{M}_{k, j_1, j_2}(\Delta t) \hat{\mathbf{z}}_k. 
\label{EASApprox}
\end{equation}
Since the only entries in the matrix $ \hat{\mathbf{z}}_k$ are $-1$
and $1$, the exact solution to the matrix equation \eqref{LkSystem}
can be attained by the scalar computation of  $\Delta \hat{M}_{k, j_1, j_2}(\Delta t)$ .

\subsubsection{Presentation of EAS}
In equations \eqref{EASApprox} and \eqref{LkSolEuler} only two cells, the ones adjacent to $k$, are updated, due to the sparsity of $L_k$. Without loss of generality, let the mass in the first cell be $m_{j_1}(t)$ and the second be $m_{j_2}(t)$, and let the first nonzero entry in $\hat{\mathbf{z}}_k$ be $+1$ and the second be $-1$. Then the evolution in \eqref{EASApprox} is,
\begin{equation}
\begin{split}
&m_{j_1}(t+\Delta t) = m_{j_1}(t) +  \Delta \hat{M}_{k, j_1, j_2}(\Delta t) \\
& m_{j_2}(t+\Delta t) = m_{j_2}(t) -  \Delta \hat{M}_{k, j_1, j_2}(\Delta t) .
\label{EAS_indiv}
\end{split}
\end{equation}
Analagously to \eqref{del_m}, the rule for $\delta m$ for EAS is,
\begin{equation}
\delta m = 
  \hat{M}_{k, j_1, j_2}(\Delta t) ,
\label{del_m_EAS}
\end{equation}
note that we do not have a second case for when $\Delta t = T - t_k$, since the function $\Delta \hat{M}_{k, j_1, j_2}(\Delta t)$ calculates an appropriate value for $\delta m$ in this case.

It can be shown that $\Delta \hat{M}_{k, j_1, j_2}(\Delta t) \leq \Delta M$, satisfying a stipulation of the schemes stated at the start of \secref{bassec}. Note that from the construction of $L_k$, we have that $|f_k(t)| A_k = |(b_k x_{j_2} -a_k x_{j_1}) |$ (as the action of $L_k$ on $\mathbf{m}$ is to calculate the contribution from face $k$ to the total rate of change of $\mathbf{m}$, which is the the flux times the area). Comparing this with \eqref{dme}, we have that 
\begin{equation}
 \Delta \hat{M}_{k, j_1, j_2}(\Delta t)= \Delta M \varphi_1(- \Delta t (a_k+b_k)).
\label{dme2}
\end{equation}
Note that the argument of $\varphi_1$ in \eqref{dme2} is that it is negative; it can be shown that for $x < 0$ then $\varphi_1(x)$ positive and less than one, so that $\Delta \hat{M}_{k, j_1, j_2}(\Delta t) \leq \Delta M$ as desired when $\Delta t$ is chosen as described. 

\section{Event Ordering and Full Algorithm For BAS and EAS}
\label{algsec}
Each face $k$ in the discretisation of $\Omega$ possesses an individual time. Using \eqref{utime2}, a timestep $\Delta t$ can be calculated for each face. We define the update time of a face $k$ as the sum of $t_k$ and its current calculated timestep,
\begin{equation}
\hat{t}_k = 
\begin{cases}
&t_k + \frac{\Delta M }{|f_k| A_k} \mbox{  if this $ \leq T$} \\
& T \mbox{  otherwise.}
\end{cases}
\label{utime3}
\end{equation}
In implementation we we keep track of $\hat{t}_k$ for every face. The face for the next event is simply chosen as the one with the smallest $t_k$. There are two reasons why a face may have a respectively smaller update time - its actual time $t_k$ may be smaller, indicating it has had on average less opportunity to update - and it may have a small calculated timestep, indicating a greater rate of activity (flux) on that face. Heuristically, both reasons would suggest that a face with smaller update time should be given greater priority for events. \\
The faces all need to be synchronised to a final time $T$ at the end of the solve, so the calculated timestep is not used when it would prevent this. Simply, if the update time of a face is greater than $T$, then $T$ is used instead. This is the second case in equations \eqref{utime2} and \eqref{utime3}. For the scheme EAS, the modified $\Delta t$ automatically results in an appropriate value of $\delta M$ via \eqref{dme2}. 

After a face is selected for an event, the amount of mass to transfer $\delta M$ is calculated by \eqref{del_m} or \eqref{del_m_EAS}, and the correct direction of transfer is calculated from the sign of the flux. After the transfer, the time $t_k$ of the face $k$ is updated to $\hat{t}_k$. Since the calculated flux of all the faces of the two cells adjacent to face $k$, including $k$ itself, are then changed, so are the calculated timesteps for each of these faces, and so are the update times. Thus, after an event, these values are re-calculated for the set of associated faces of $k$, and the cycle then repeats. To make the process of finding the face with lowest $\hat{t}_k$ efficient, an editable priority queue is used; see the appendix of \cite{mythesis}.


\begin{algorithm}
\caption{Pseudo code for the asynchronous scheme BAS/EAS}
\label{alg1}
\begin{algorithmic}[1]
\STATE{Data: Grid structure, Initial concentration values, $\Delta M$, $T$}
\STATE{Initialise: $t =0$ ; Calculate $f_l$ from \eqref{fd_ad_flux} and $\hat{t}_l $ from \eqref{utime3} $\forall \mbox{ faces } k $ }
\WHILE{$t \leq T$}

\STATE{Find face $k$ s.t. $\hat{t}_k = \min_{l \in \mathcal{F}}{\hat{t}_l}$  }
\STATE{Get cells $j_1$ and $j_2$ adjacent to $k$}
\STATE{$\Delta m = |f_k| (\hat{t}_k - t_k) A_k$ }
\STATE{Calculate $\delta m$ using \eqref{del_m} or \eqref{del_m_EAS}}
\STATE{  $m_{j_1} \leftarrow m_{j_1}- \mbox{sign} (f_k)\delta m$ }
\STATE{ $m_{j_2} \leftarrow m_{j_2}+ \mbox{sign} (f_k)\delta m$ }
\STATE{ $t = t_k \leftarrow  \hat{t}_k$ }
 \FOR{ $l \in \mathcal{\tilde{F}}_k$}
\STATE{Recalculate $f_l$ from \eqref{fd_ad_flux} }
\STATE{  Recalculate $\hat{t}_l$ from \eqref{utime3} }
\ENDFOR
\STATE{Choose $k$ s.t. $\hat{t}_k = \min_{\mbox{faces }l}{\hat{t}_l}$  }
\ENDWHILE
\RETURN $T$
\end{algorithmic}
\end{algorithm}

\algref{alg1} describes the general asynchronous method. After initialising the required values on all faces, the update loop is run until every face is synchronised to the desired final time of $T$. Each iteration of the loop is a single event and proceeds as follows. First the face with the lowest projected update time $\hat{t}$ is found (line 3) Then the two cells adjacent to this face are located from the grid structure (line 4). The amount of mass to transfer between these cells is calculated (line 5). This equation simply returns the global mass unit $\Delta M$ in most cases, except when the face is being forced to use an update time $T$; see equation (\ref{utime3}). Mass is transferred between the cells in the correct direction (lines 6-7). A loop (lines 8-12) updates the faces of cells $j_1$ and $j_2$; recalculating their fluxes and update times based on the new mass values. The loop continues by finding the next face with the lowest uptime (back to line 3). \\

\section{Numerical Results}
\label{numerical}

For the first two experiments we solve the advection-diffusion equation, 
\begin{equation}
\frac{dc(\mathbf{x},t)}{dt} =  \nabla^2 D(\mathbf{x}) (c(\mathbf{x},t)) + \nabla \mathbf{v}(\mathbf{x}) c(\mathbf{x},t).
\label{ad_dif_full_c3}
\end{equation} 
For the third experiment the we add a reaction term to \eqref{ad_dif_full_c3}. Our new schemes were implemented in C++ and used in Matlab through the Mex interface. The Matlab Reservoir Simulation Toolkit (MRST \cite{mrst_prime}) was used to generate the grids for the experiments but the discretisation and solver routines were implemented by us. For the comparison solves exponential integrators were used \cite{Overview, hochbruck-rkei-2005, cm, tambue2010exponential, tambue2013efficient, NW}, which for linear systems like the advection-diffusion equation are known to be effectively exact. Error was measured using the discrete approximation of the $L^2$ norm. Timings were performed using Matlab's tic and toc commands and thus the units of cputime below are all in seconds.

\subsection{High Peclet number fracture}
\label{fracsec}

In this example a single layer of cells is used, making the problem effectively two dimensional. The domain is $\Omega = 10 \times 10 \times 10$ metres, divided into $100 \times 100$ cells of equal size. Thus each cell has volume $0.1 m^3$ and each internal face has area $1 m^2$. The PDE to be solved is \eqref{ad_dif_full_c3}, and the diffusivity and velocity fields were prepared as follows. \\
A fracture in the domain is represented by having a line of cells which we give certain properties. These cells were chosen by a weighted random walk through the grid (weighted to favour moving in the positive $y$-direction so that the fracture would bisect the domain). This process started on an initial cell which was marked as being in the fracture, then randomly chose a neighbour of the cell and repeated the process. In this example, the fracture cells differs from the rest of the domain in that they have different permeability values. This has an effect on the pressure and thus velocity fields in the domain, as given by the standard Darcy's Law relations. 
We use a permeability matrix of the form
$$
\mathbf{K} = 
\left(
\begin{array}{cc}
k_x & 0 \\
0 & k_y
\end{array}
\right).
$$
We set $k_x = 1$ in all cells except on the $x=10$ boundary where it was set to zero, and $k_y = 2000$ on the fracture cells and $k_y = 1$ in all other cells, except on the $y=10$ boundary where it was set to zero. See \figref{frac-flow-setup} a).
This is intended to cause a large velocity in the y-direction inside the fracture, and a small velocity elsewhere. The steady-state Darcy equation is used to find the pressure field. 
Dirchlet boundary conditions were imposed; with $p(x,0) = 1$, $p(x,10) = 0$, and no-flow conditions on the other edges. That is, high pressure along the $y=0$ edge of the domain and low pressure at the $y=10$ edge of the domain, creating a gradient. We then approximate the solution to the steady state Darcy equation (a finite volume discretisation was used to produce a linear system which was then solved). The resulting pressure field can be seen in \figref{frac-flow-setup} b). 
We used a finite difference approximation of the Darcy equation. 
(with $\phi$, $\mu$ set to one and $\mathbf{g}$ set to zero) to find the resulting velocity field. The x- and y- velocities can be seen in \figref{frac-flow-setup} c) and d), on a logarithmic scale. 
The resulting y-velocity is extremely high in the fracture compared to elsewhere in the domain; this is expected to produce highly localised activity.  Some streamlines are shown in \figref{frac-flow-setup} e); here the flow is in the y-direction, i.e. bottom to top in the plot. \\
In every cell we set the diffusivity $D = 0.01$. A measure of the relative importance of advection compared to diffusion in a cell $k$ is the P\'eclet number, $Pe_k$. We show the logarithm of the  P\'eclet number in \figref{frac-flow-setup} f). It varies by five orders of magnitude in the domain. \\
With the velocity and diffusivity fields thus prepared, we approximate the advection-diffusion equation \eqref{ad_dif_full_c3} on this domain. Zero Neumann boundary conditions (`no flow') were applied on every boundary, and the initial condition was $c(\mathbf{x})=0$ everywhere except for $\mathbf{x} = (5.95, 0.05)^T$, where $c(\mathbf{x})=1$. That is, there is a single cell with concentration $1$ on the bottom edge of the domain, close to the fracture. The final time was $T=17$. The final solution approximated with EAS using $\Delta M = 10^{-8}$ is shown in \figref{frac-flow-show} a). In \figref{frac-flow-show} b)we show heat maps of the logarithm of the number of events experienced by a cell during the solve. We consider a cell to have had an event if one of its faces has an event. These plots shows how the activity is localised, as the number of events varies in seven orders of magnitude between a large part of the domain and the fracture. 
In \figref{frac-flow-plotsMain} we show convergence and parameter relations. Plot a) shows the estimated error against $\Delta M$; we observe a roughly first order convergence of the error, as well as the more advanced EAS being generally more accurate than BAS. In plot b) we plot the (average over the whole run) timestep $\Delta t$ against the estimated error, again an approximately first order convergence is observed. Plot c) shows the total number of events over the solve $N$ against the mass unit $\Delta M$, here we observe not only a strong $N = O(\Delta M^{-1})$ relation, but also that the number of events $N$ for the two different schemes converges towards each other as $\Delta M \rightarrow 0$. This may indicate the existence of some sort of preferred path or ordering of events to which both schemes converge, however this conjecture is left until further work for exploration. Plot d) shows the average timestep $\Delta t$ against $N$ and strongly indicates that $\Delta t = O(N^{-1})$.

\begin{figure}[h]
\centering
\begin{minipage}[b]{0.45\linewidth}
a) \\
\includegraphics[width=0.99\columnwidth]{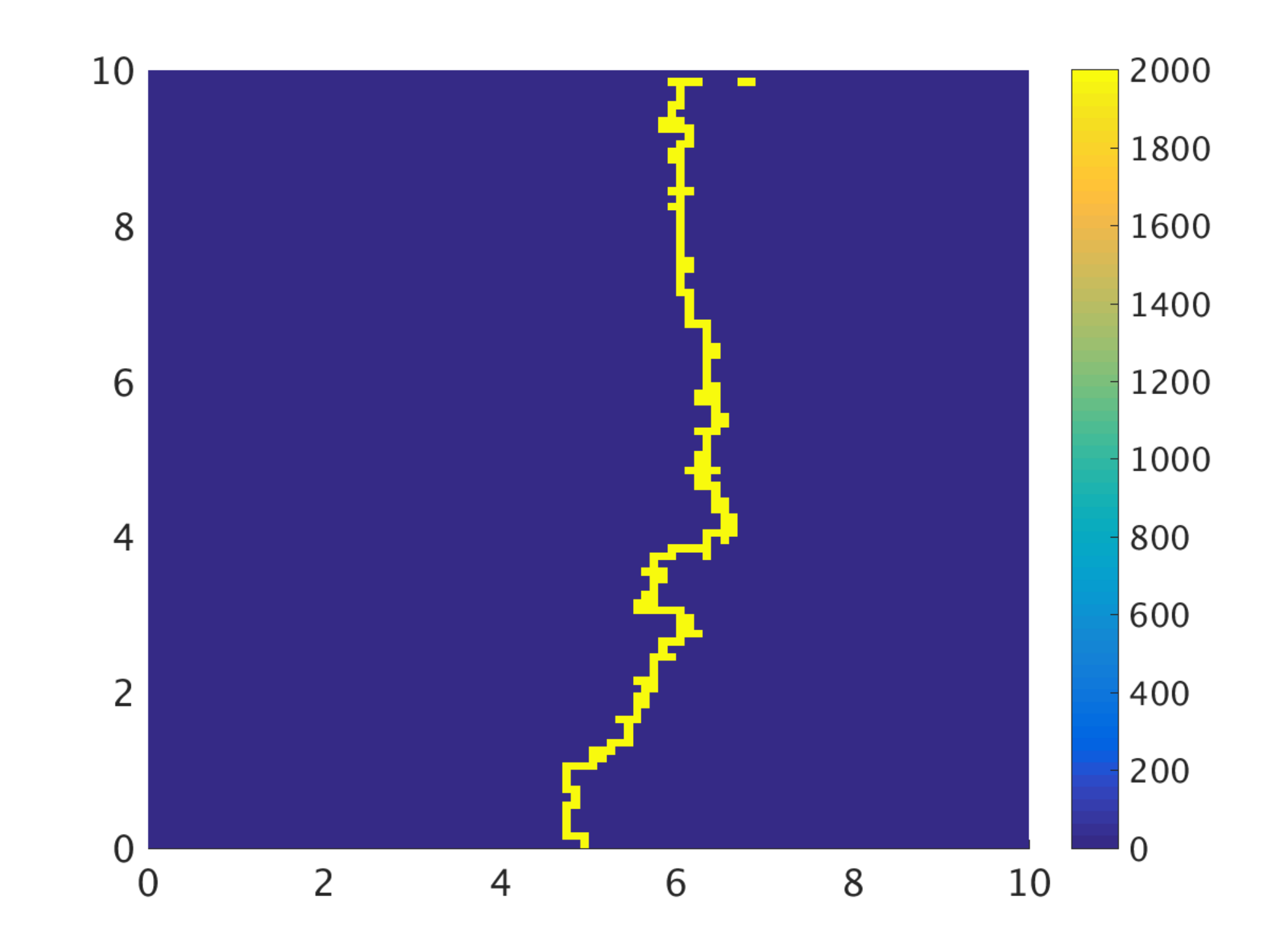}
\end{minipage}
\begin{minipage}[b]{0.45\linewidth}
b) \\
\includegraphics[width=0.99\columnwidth]{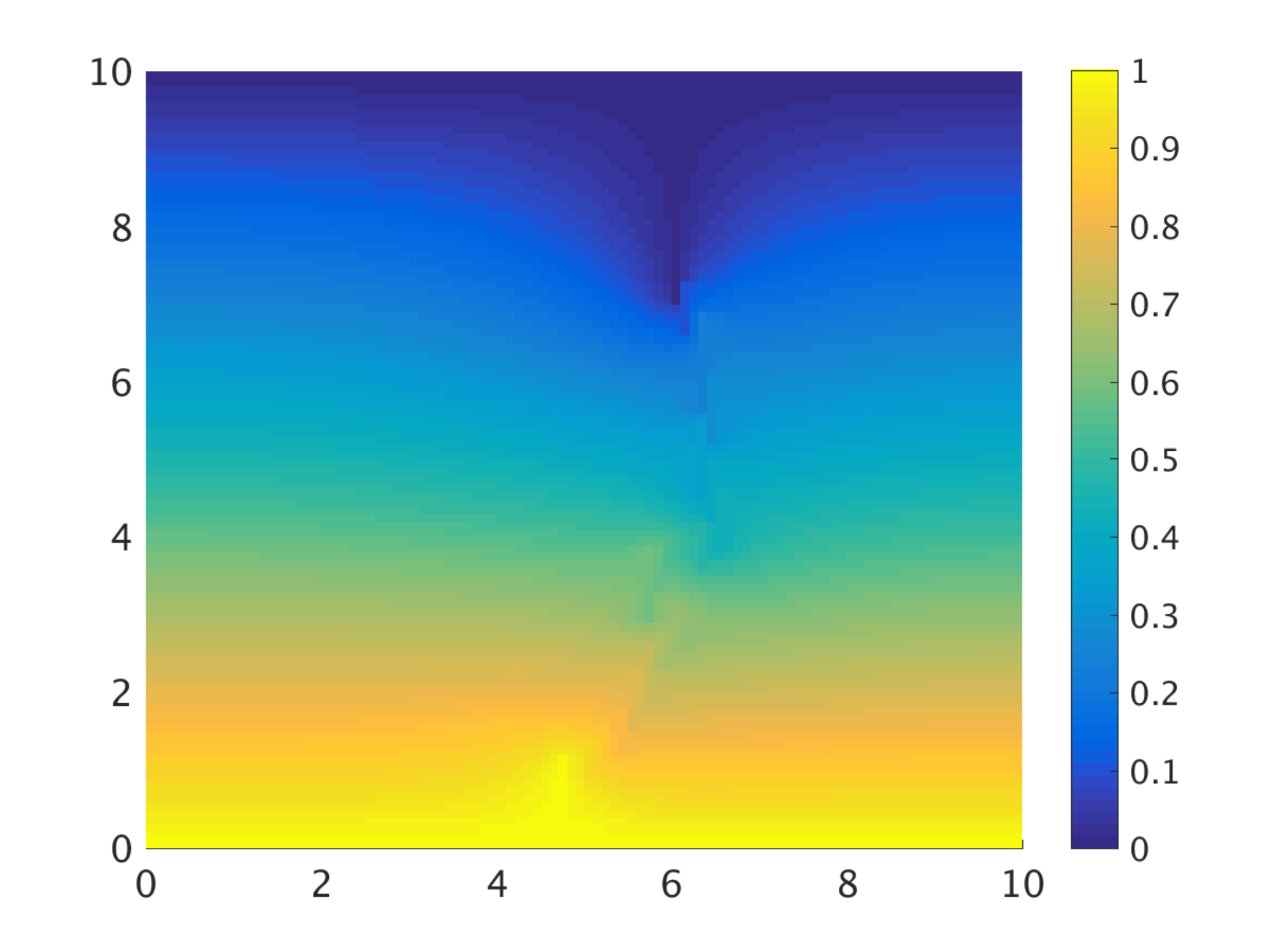}
\end{minipage}
\begin{minipage}[b]{0.45\linewidth}
c) \\
\includegraphics[width=0.99\columnwidth]{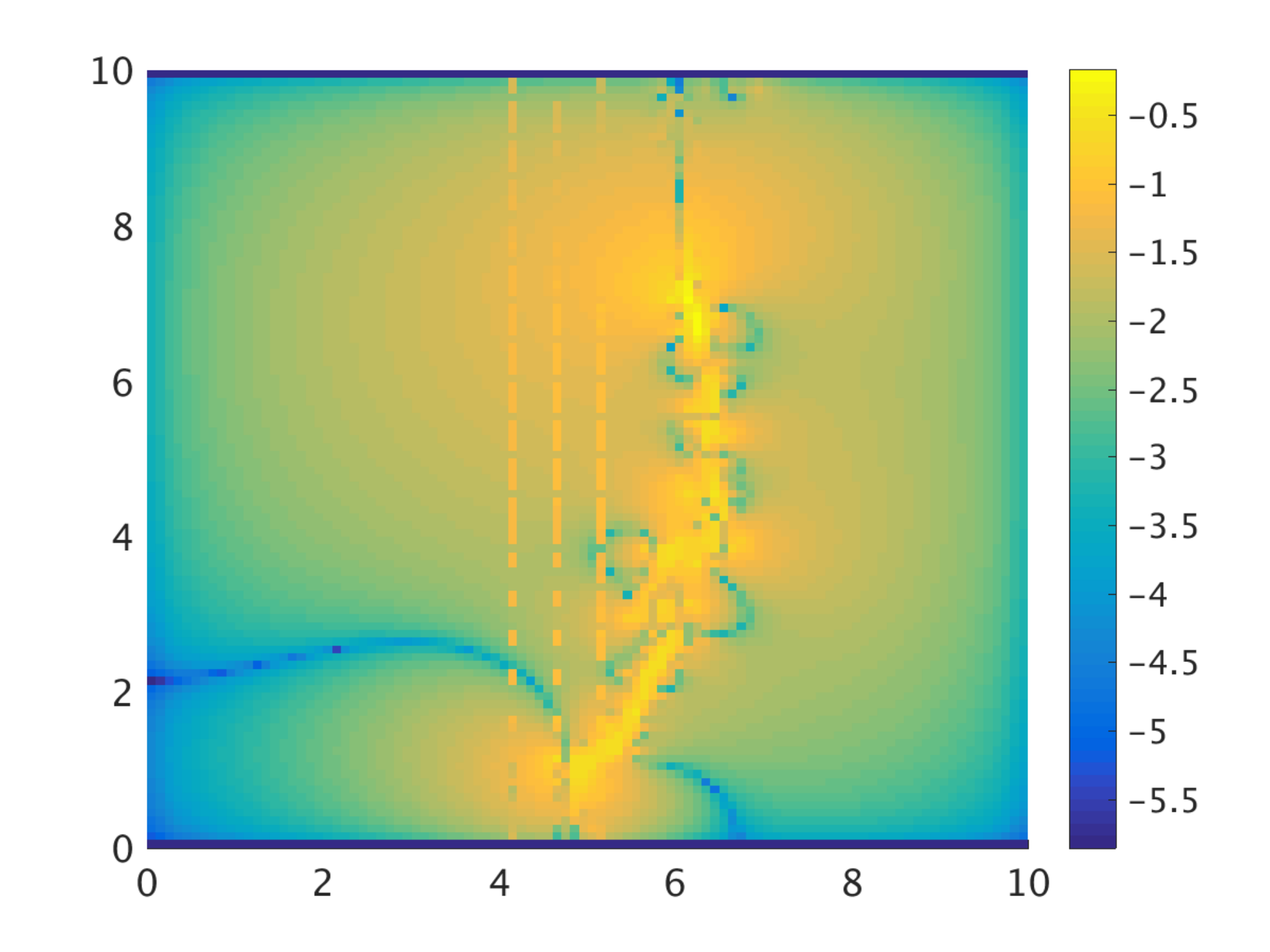}
\end{minipage}
\begin{minipage}[b]{0.45\linewidth}
d) \\
\includegraphics[width=0.99\columnwidth]{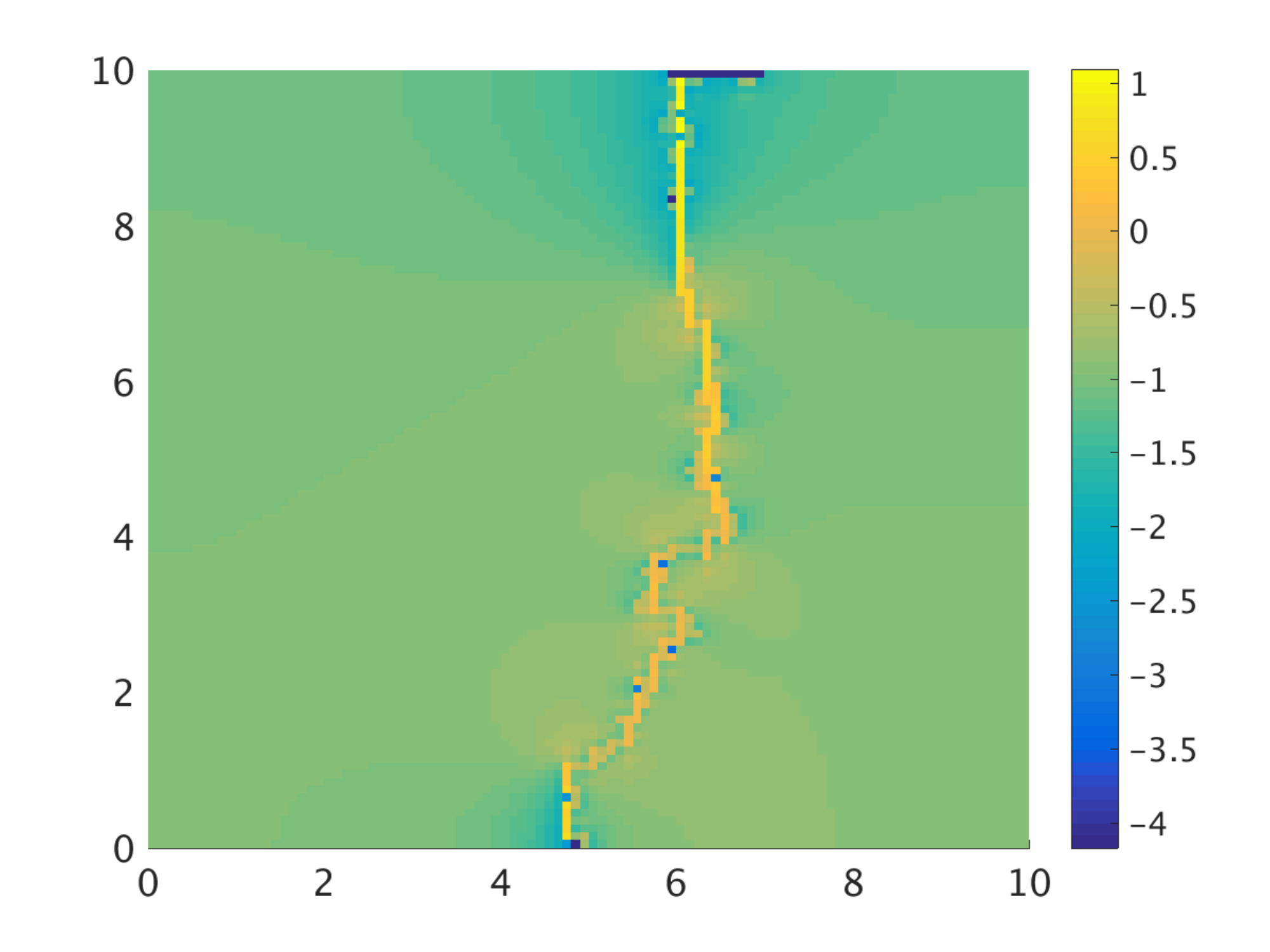}
\end{minipage}
\begin{minipage}[b]{0.45\linewidth}
e) \\
\includegraphics[width=0.99\columnwidth]{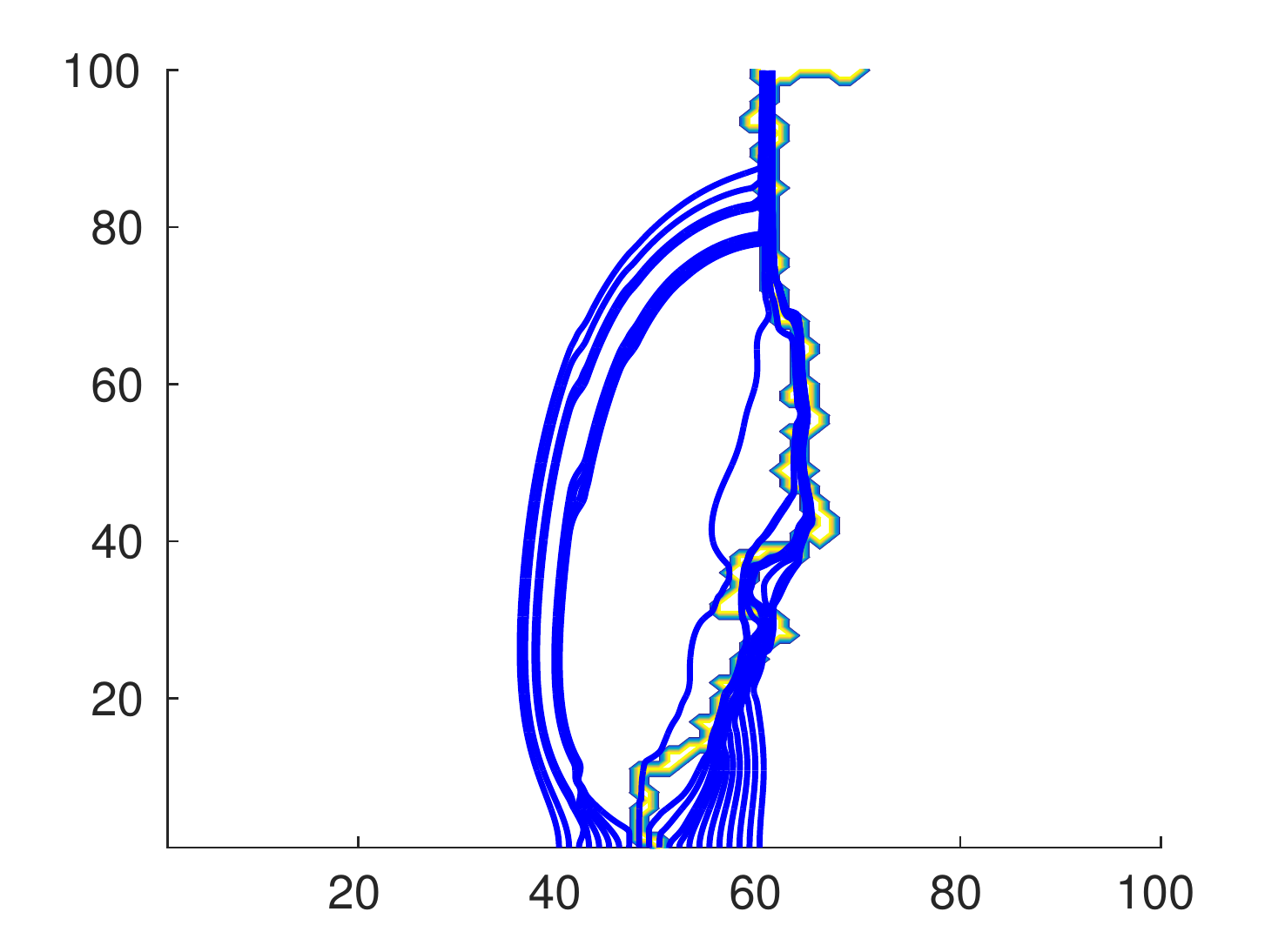}
\end{minipage}
\begin{minipage}[b]{0.45\linewidth}
f) \\
\includegraphics[width=0.99\columnwidth]{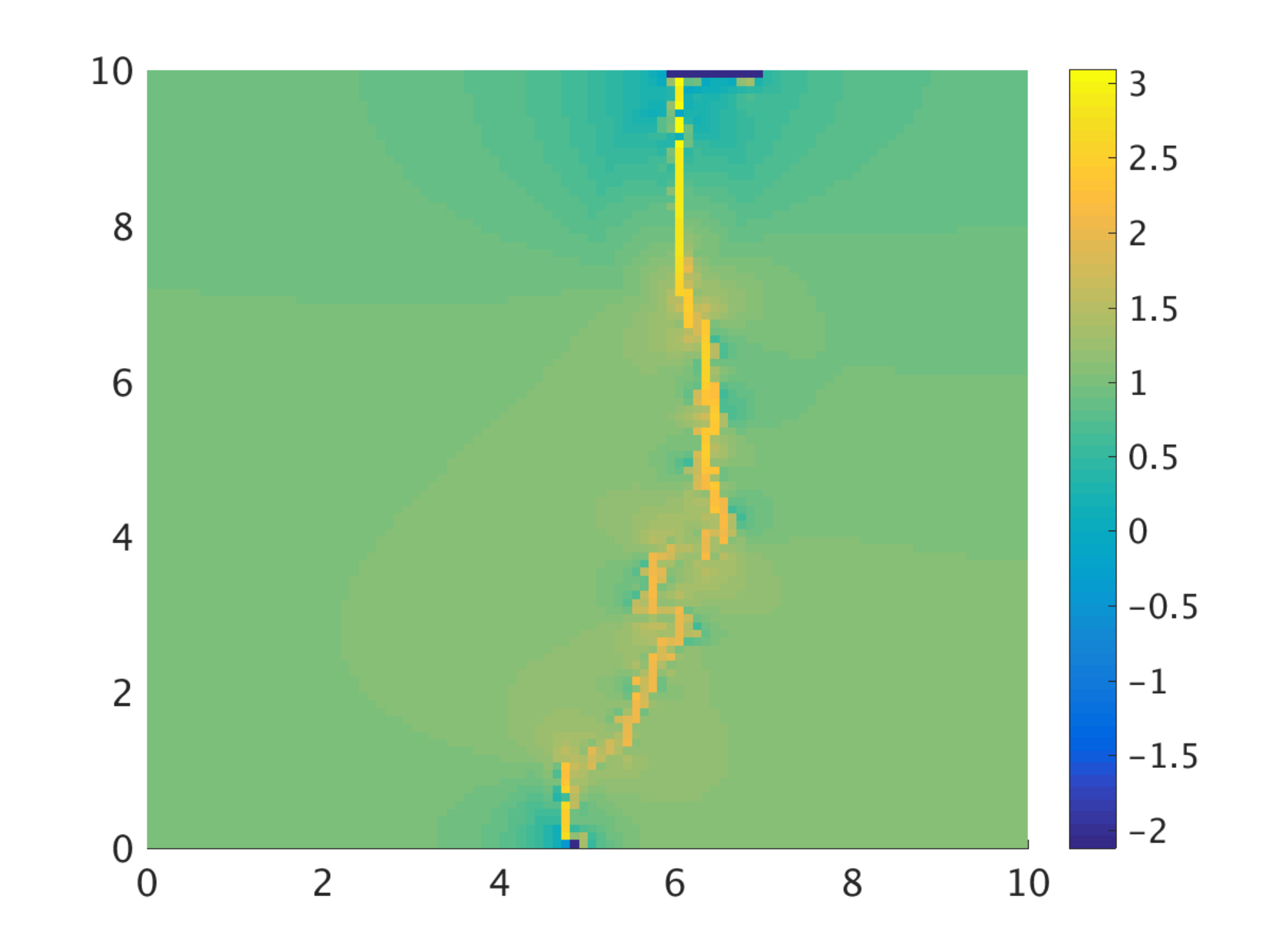}
\end{minipage}
\caption[Permeability, pressure and velocity fields for the first fracture example. ]{Permeability, pressure and velocity fields for the fracture example with varying velocity, \secref{fracsec}. a) $k_y$, the permeability in the y-direction. b) The calculated pressure field. c) The calculated velocity field in the x-direction, logarithmic scale. d) The calculated velocity field in the y-direction, logarithmic scale. e) Some streamlines of the calculated velocity field; flow is from bottom to top. f) P\'eclet number, logarithmic scale. }
\label{frac-flow-setup}
\end{figure}


\begin{figure}[h]
\centering
\begin{minipage}[b]{0.45\linewidth}
a)\\
\includegraphics[width=0.99\columnwidth]{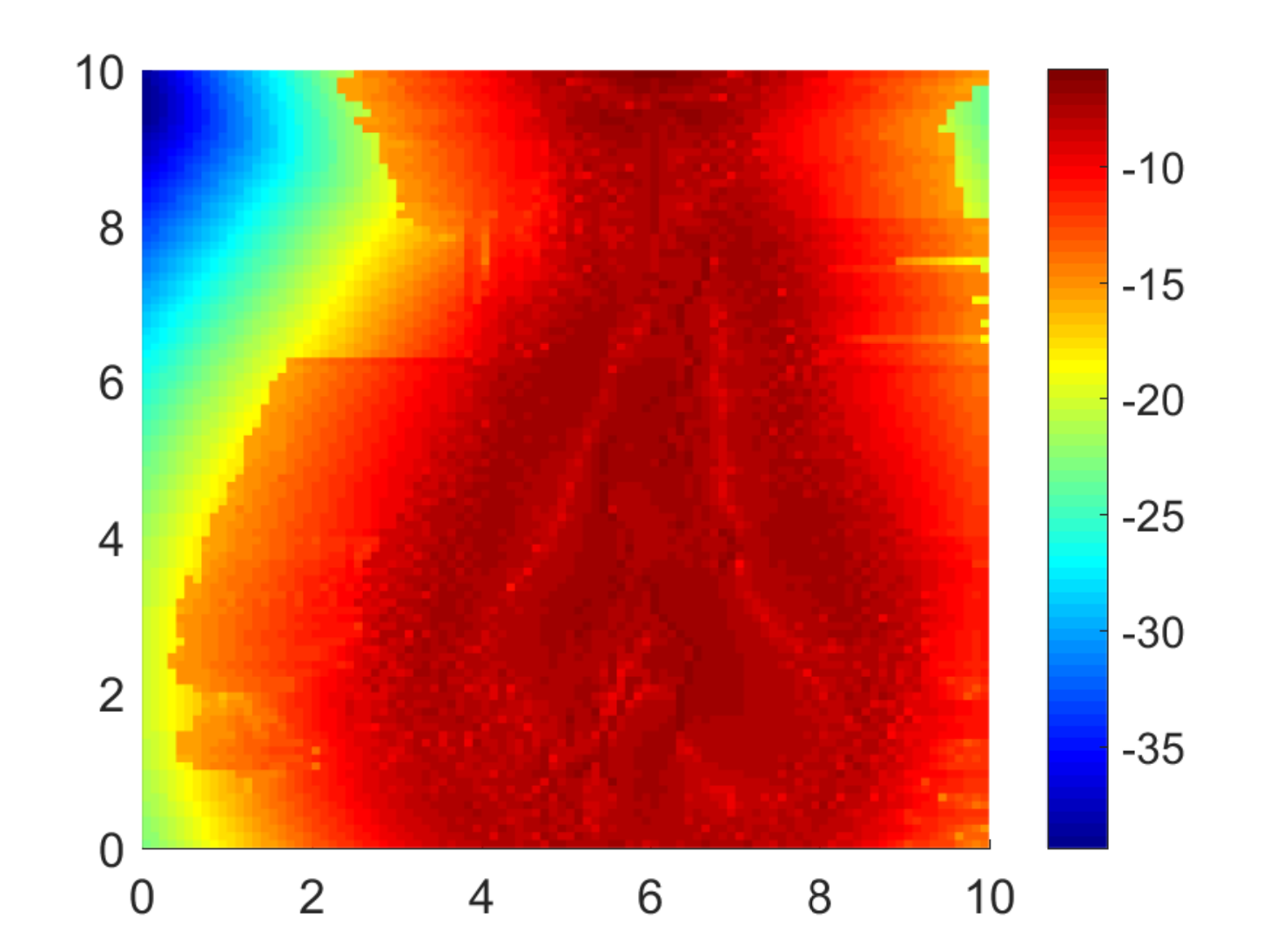}
\end{minipage}
\begin{minipage}[b]{0.45\linewidth}
b) \\
\includegraphics[width=0.99\columnwidth]{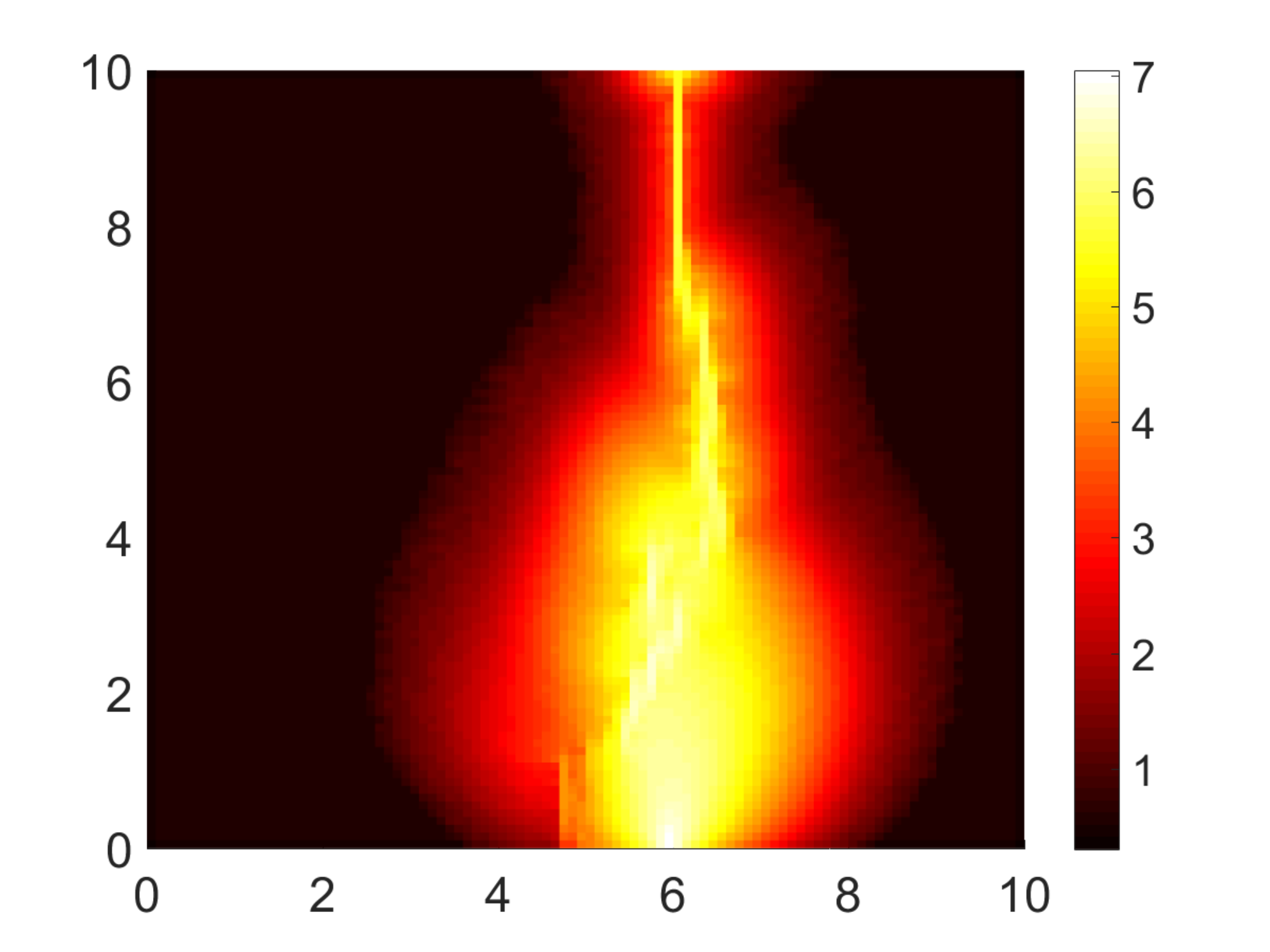}
\end{minipage} 
\caption[Final state of the first fracture example.]{Final state of the fracture example, approximated with each of the new schemes, with $\Delta M =10^{-8}$. a) Error compared to comparison solve (logarithmic). b)Events per cell, logarithmic scale.}
\label{frac-flow-show}
\end{figure}

\begin{figure}[h]
\centering
\begin{minipage}[b]{0.45\linewidth}
a) \\
\includegraphics[width=0.99\columnwidth]{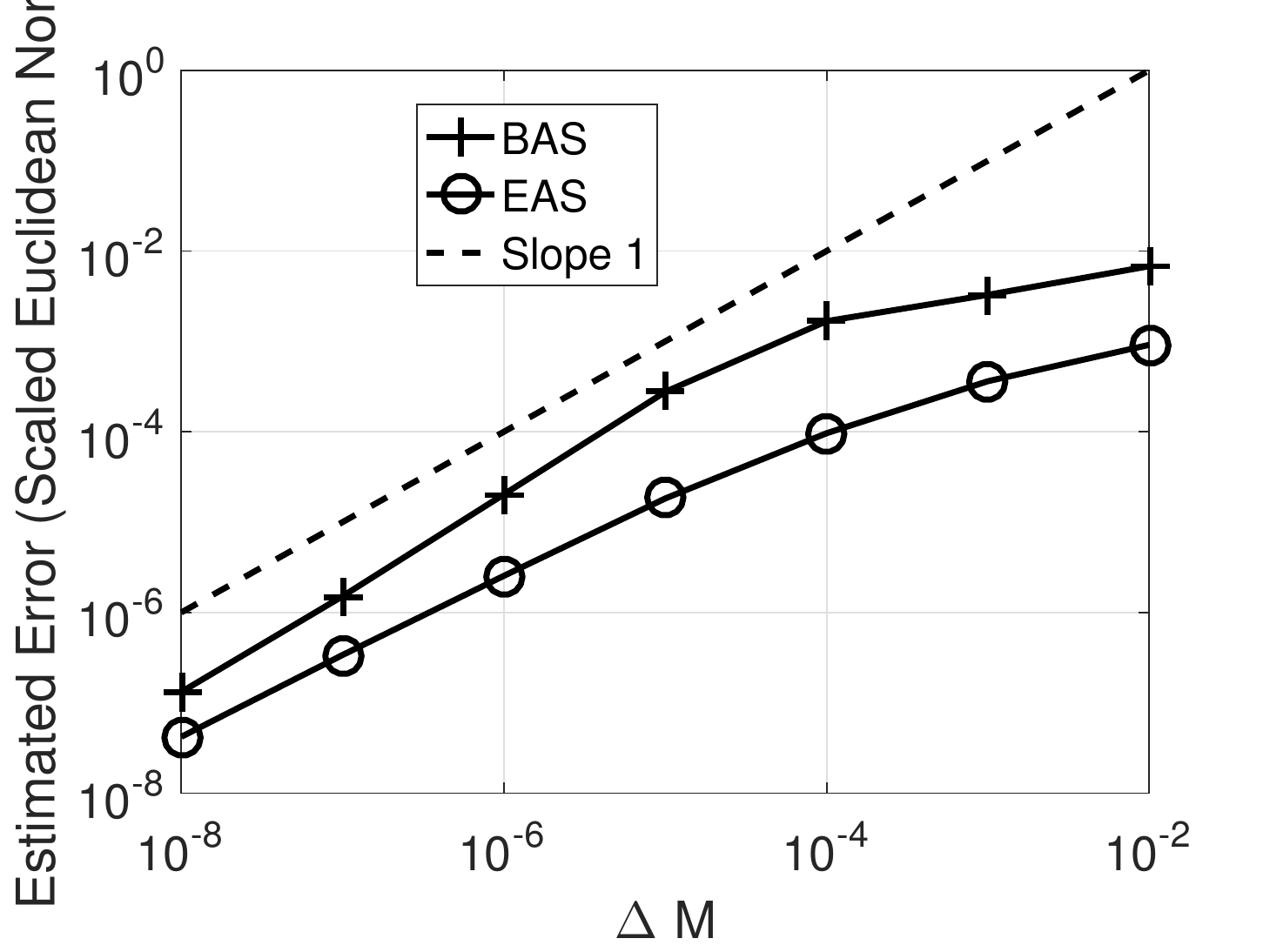}
\end{minipage}
\begin{minipage}[b]{0.45\linewidth}
b) \\
\includegraphics[width=0.99\columnwidth]{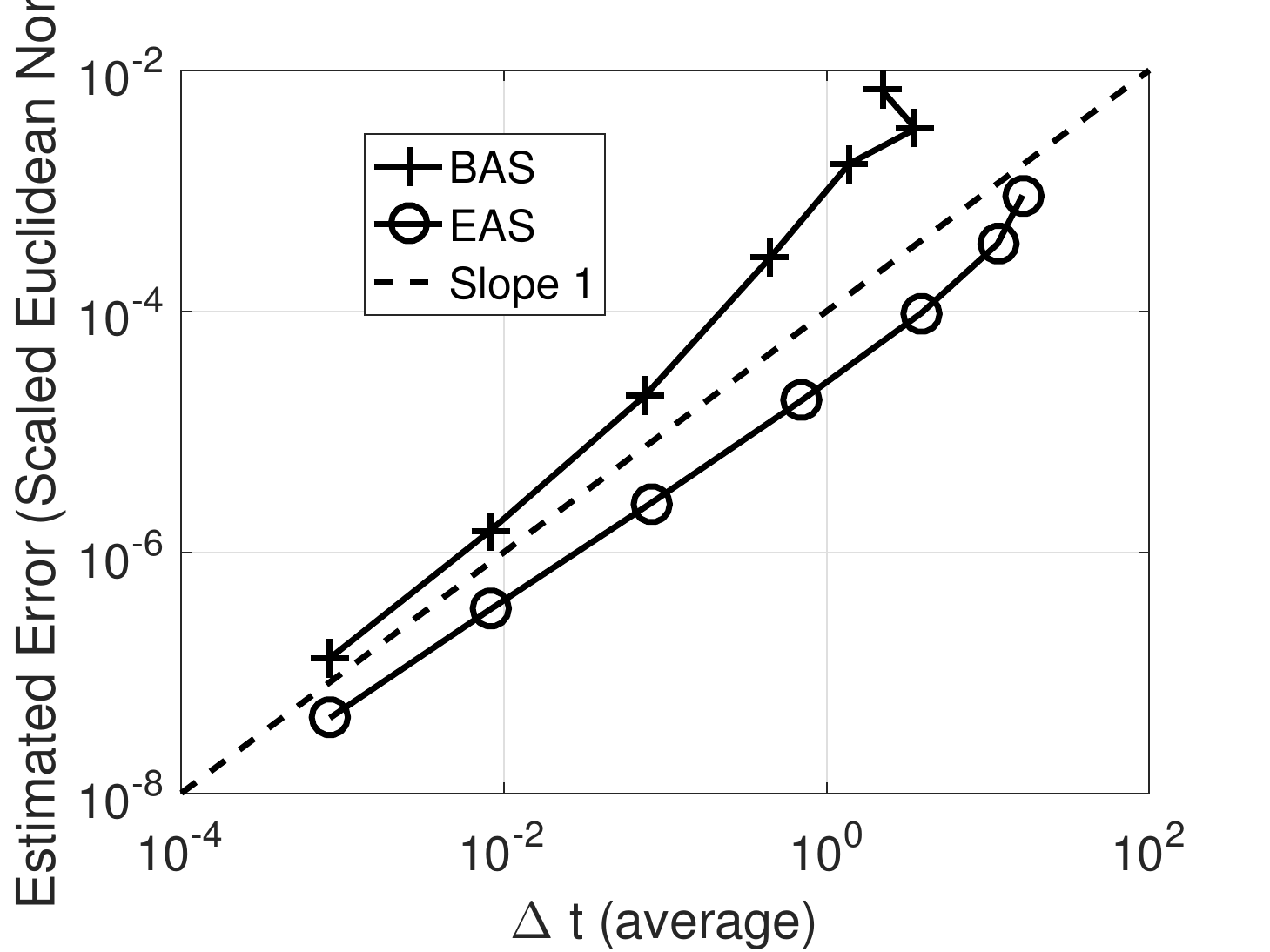}
\end{minipage} 
\begin{minipage}[b]{0.45\linewidth}
c) \\
\includegraphics[width=0.99\columnwidth]{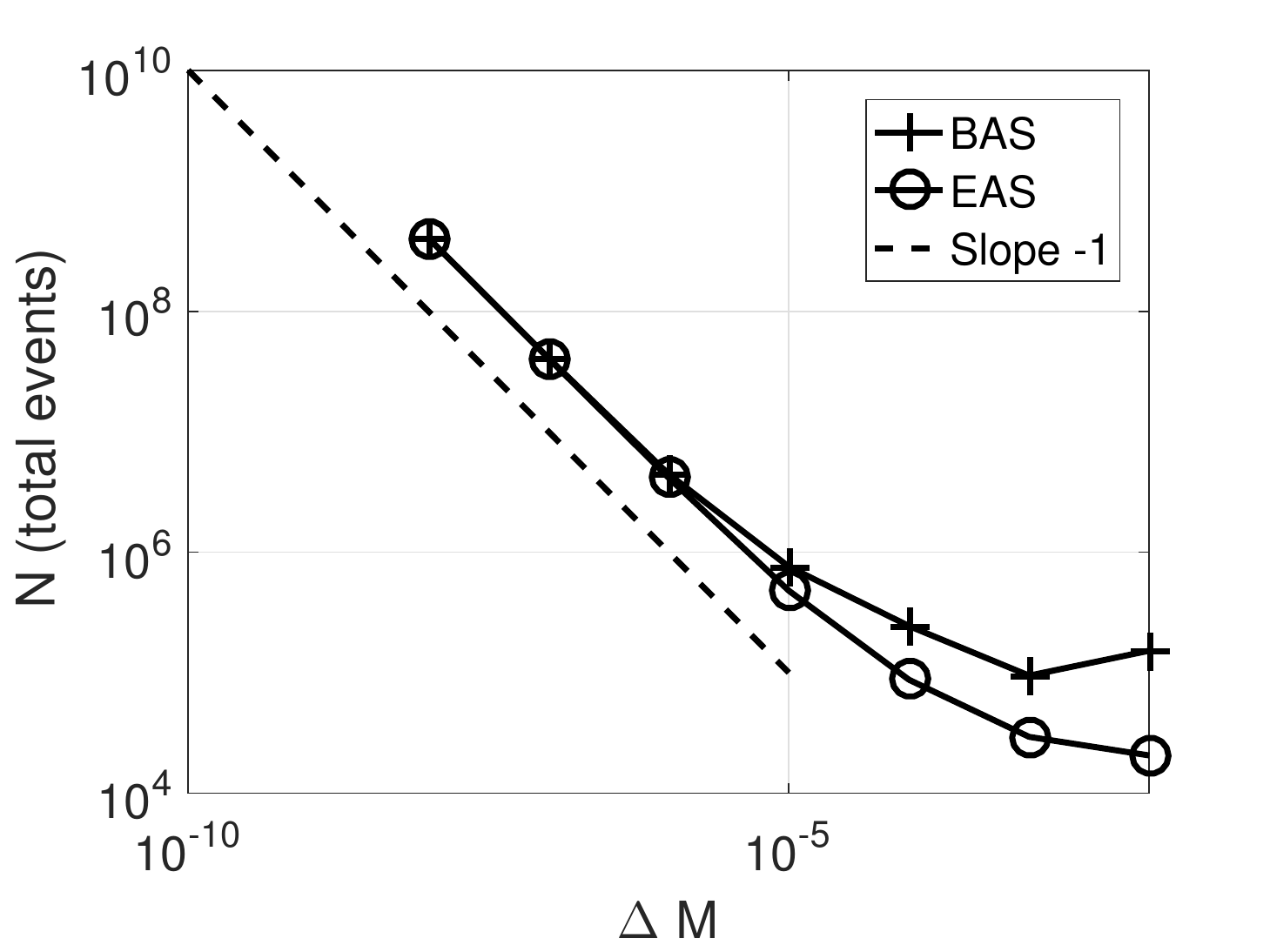}
\end{minipage}
\begin{minipage}[b]{0.45\linewidth}
d) \\
\includegraphics[width=0.99\columnwidth]{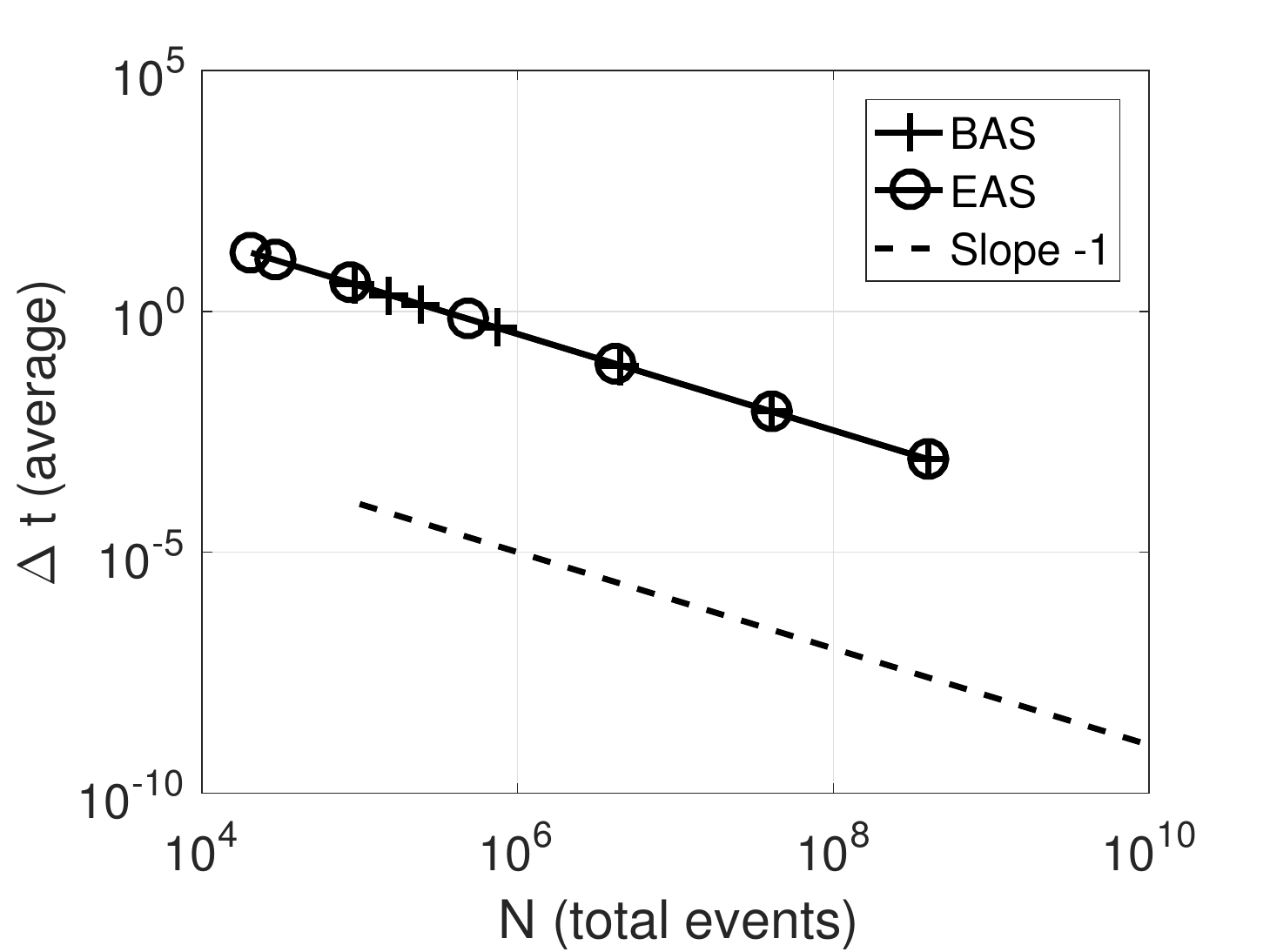}
\end{minipage} 
\caption{Results for the experiment described in \secref{fracsec}. }
\label{frac-flow-plotsMain}
\end{figure}

\clearpage

\subsection{Random diffusivity field}
\label{sf_sec}
This test is in two dimensions, with a random diffusivity field. The PDE is again \eqref{ad_dif_full_c3}. The domain is again $\Omega = 10 \times 10 \times 10$ metres and discretised into $100 \times100 \times 1$ cells. The diffusivity field is as follows. We prepared a Normal, mean-zero, random field $\psi(\mathbf{x})$ over the cells with correlation function
$$
C(X,Y) = e^{\frac{||X-Y||}{l}},
$$
with the correlation length between the x and y directions being $l=9$. The diffusivity field used was then $D(\mathbf{x})= 10.0^{\psi(\mathbf{x})}$. The field was generated using the standard Cholesky technique (see for example \cite{lord2014introduction}); there was no need for approximation due to the relatively small number of cells.  The velocity field was uniformly zero.\\
For this test the concentration was $c(\mathbf{x})=0$ for all $\mathbf{x}$ except at $\mathbf{x} = (4.95, 5.05)^T$ where $c(\mathbf{x})=0$. The boundary conditions were no-flow on all boundaries. \\
In \figref{sf_show} is displayed the comparison solve (produced by the exponential integrator), and solves with EAS with different values of $\Delta M$, showing how decreasing $\Delta M$ increases the agreement of the solve with EAS to the comparison solve. Also in \figref{sf_show} d) is a plot showing the number of events across the system, showing where activity was concentrated. \\
In \figref{sf_plotsMain} we show convergence results and parameter relations for this system. Broadly, most of the conclusions are the same as from the previous experiment; see the discussion in \secref{frac_sec}. \\

\begin{figure}[h]
\centering
\begin{minipage}[b]{0.45\linewidth}
a) \\
\includegraphics[width=0.99\columnwidth]{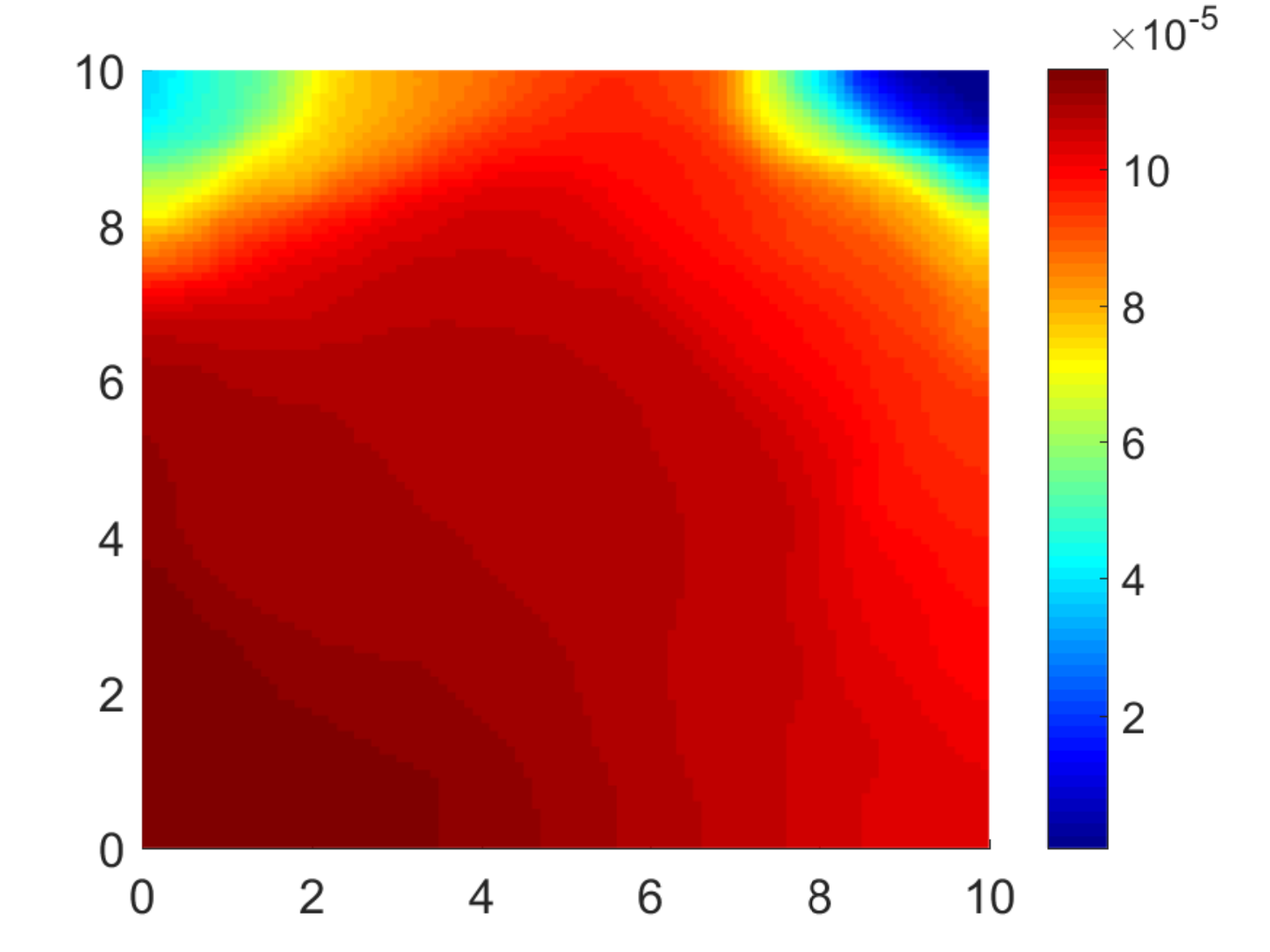}
\end{minipage}
\begin{minipage}[b]{0.45\linewidth}
b) \\
\includegraphics[width=0.99\columnwidth]{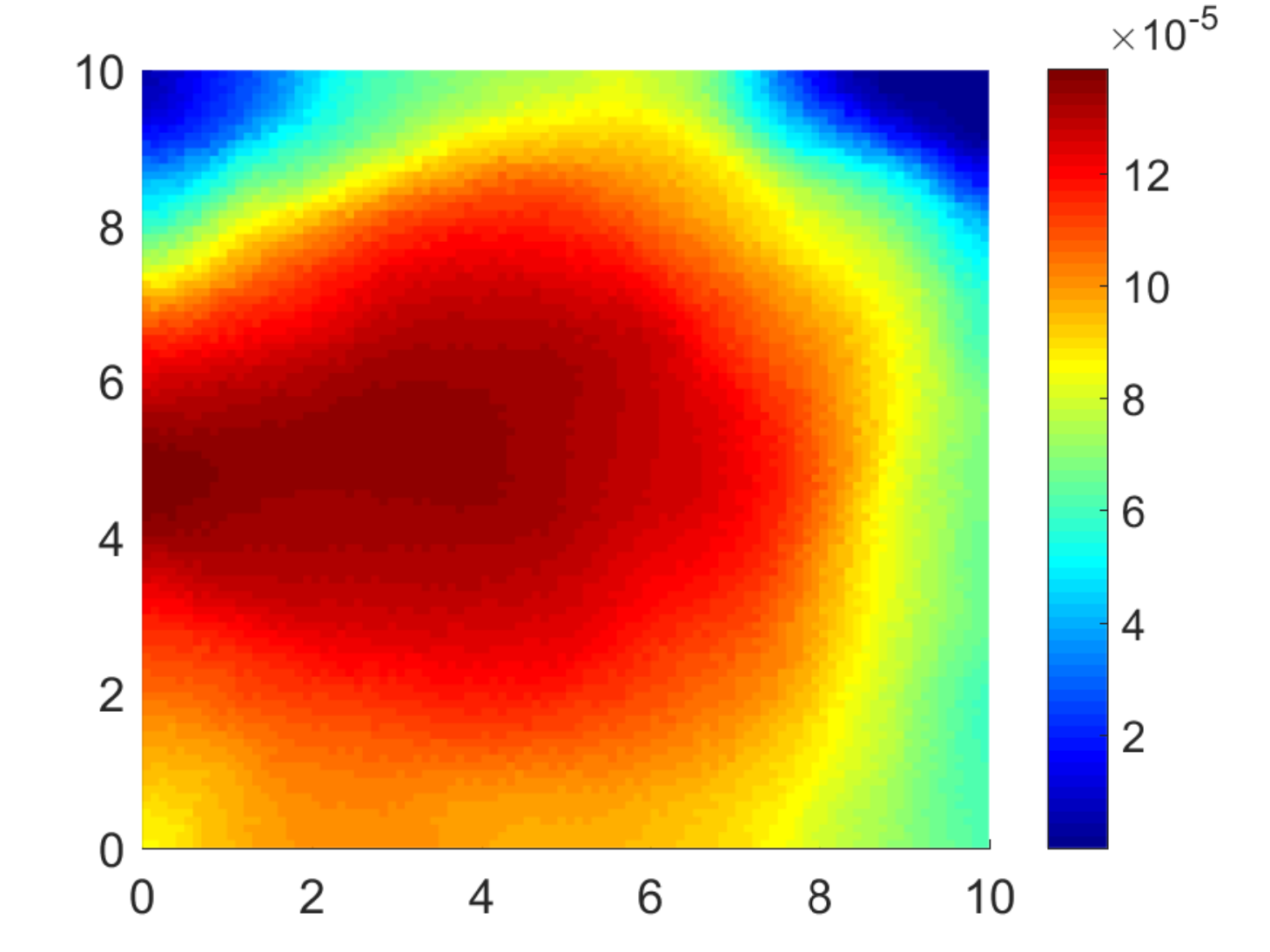}
\end{minipage} 
\begin{minipage}[b]{0.45\linewidth}
c) \\
\includegraphics[width=0.99\columnwidth]{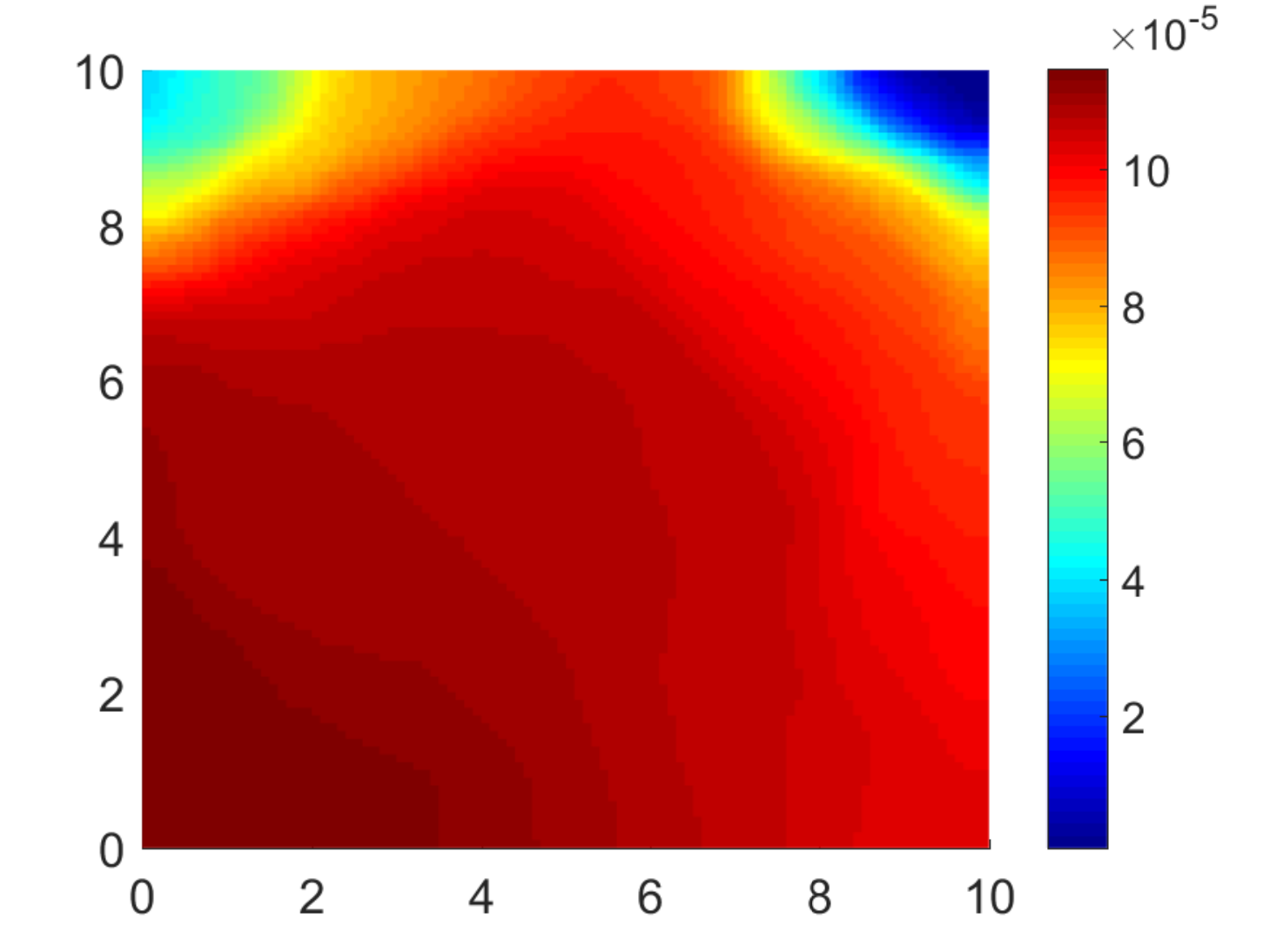}
\end{minipage}
\begin{minipage}[b]{0.45\linewidth}
d) \\
\includegraphics[width=0.99\columnwidth]{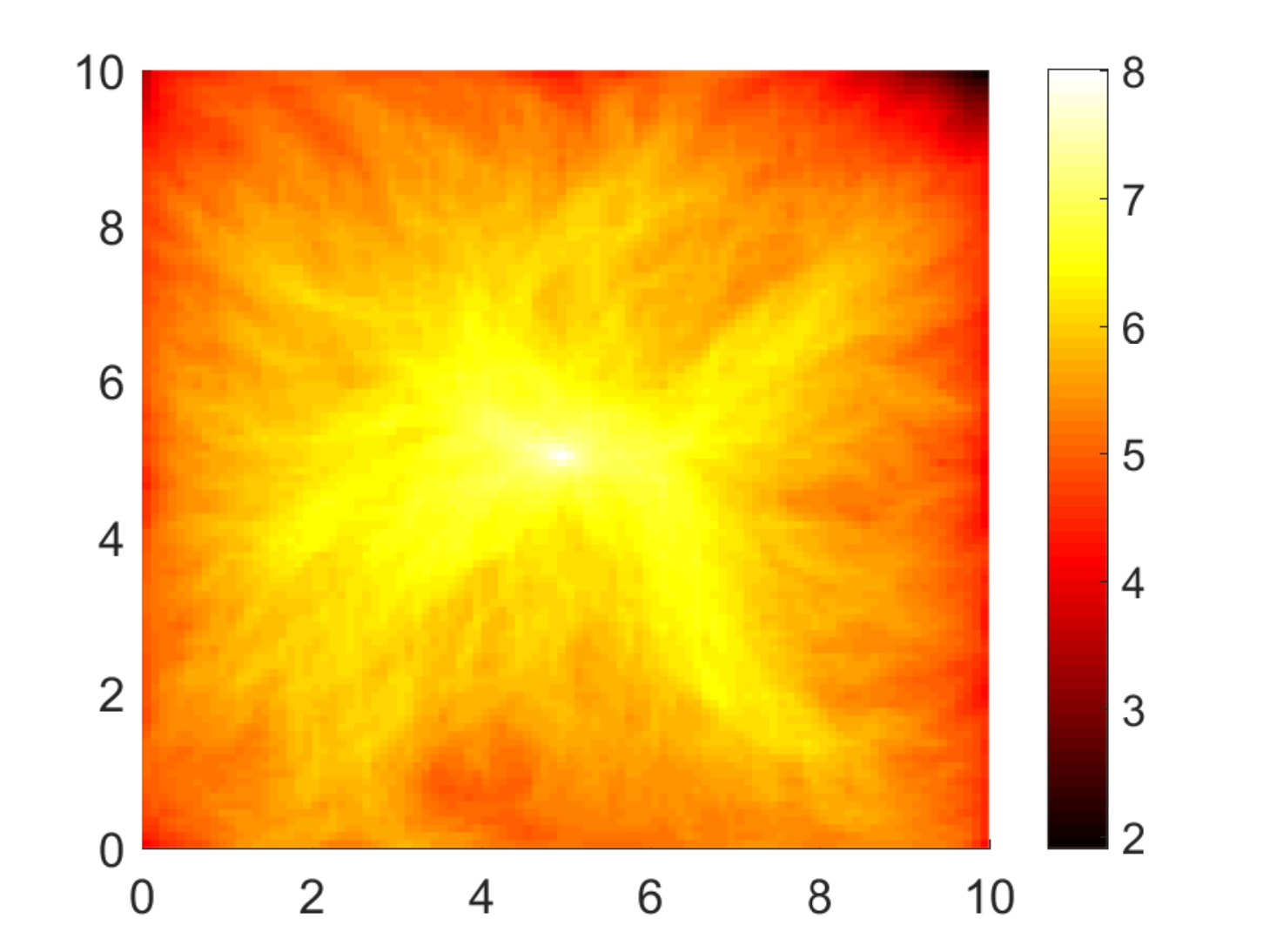}
\end{minipage} 
\label{sf_show}
\caption{For the system described in \secref{sf_sec}. a) The comparison solve produced with an exponential integrator. b) Solution produced by EAS with $\Delta M = 10^{-6}$; here $\Delta M$ is too great to allow sufficient accuracy (although positivity is preserved). c) Solution produced by EAS with $\Delta M = 10^{-9}$; this solution is in strong agreement with the comparison solve. d) Shows logarithm of number of events experienced by each cell for the run with EAS and $\Delta M = 10^{-9}$.}
\end{figure}

\begin{figure}[h]
\centering
\begin{minipage}[b]{0.45\linewidth}
a) \\
\includegraphics[width=0.99\columnwidth]{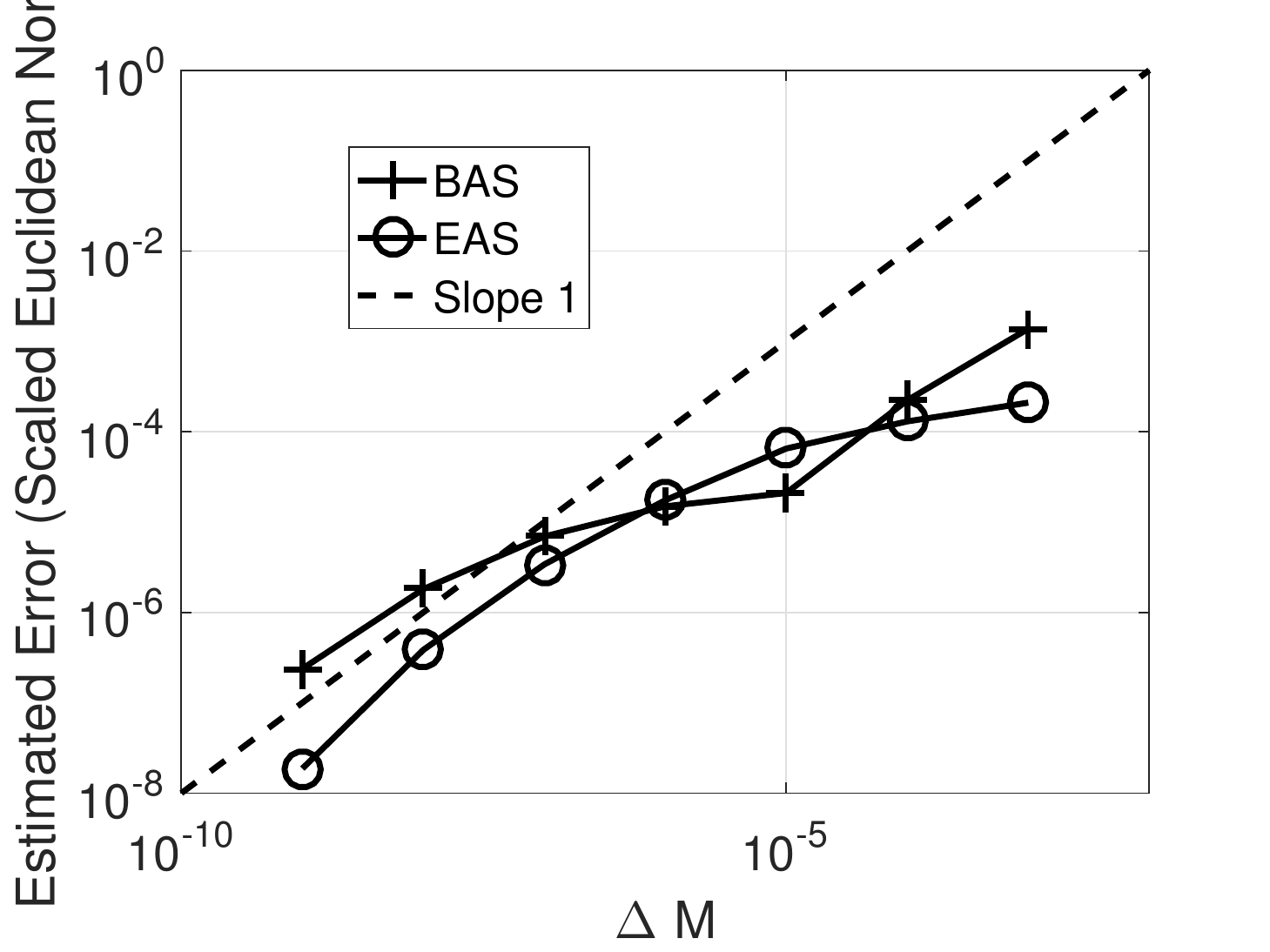}
\end{minipage}
\begin{minipage}[b]{0.45\linewidth}
b) \\
\includegraphics[width=0.99\columnwidth]{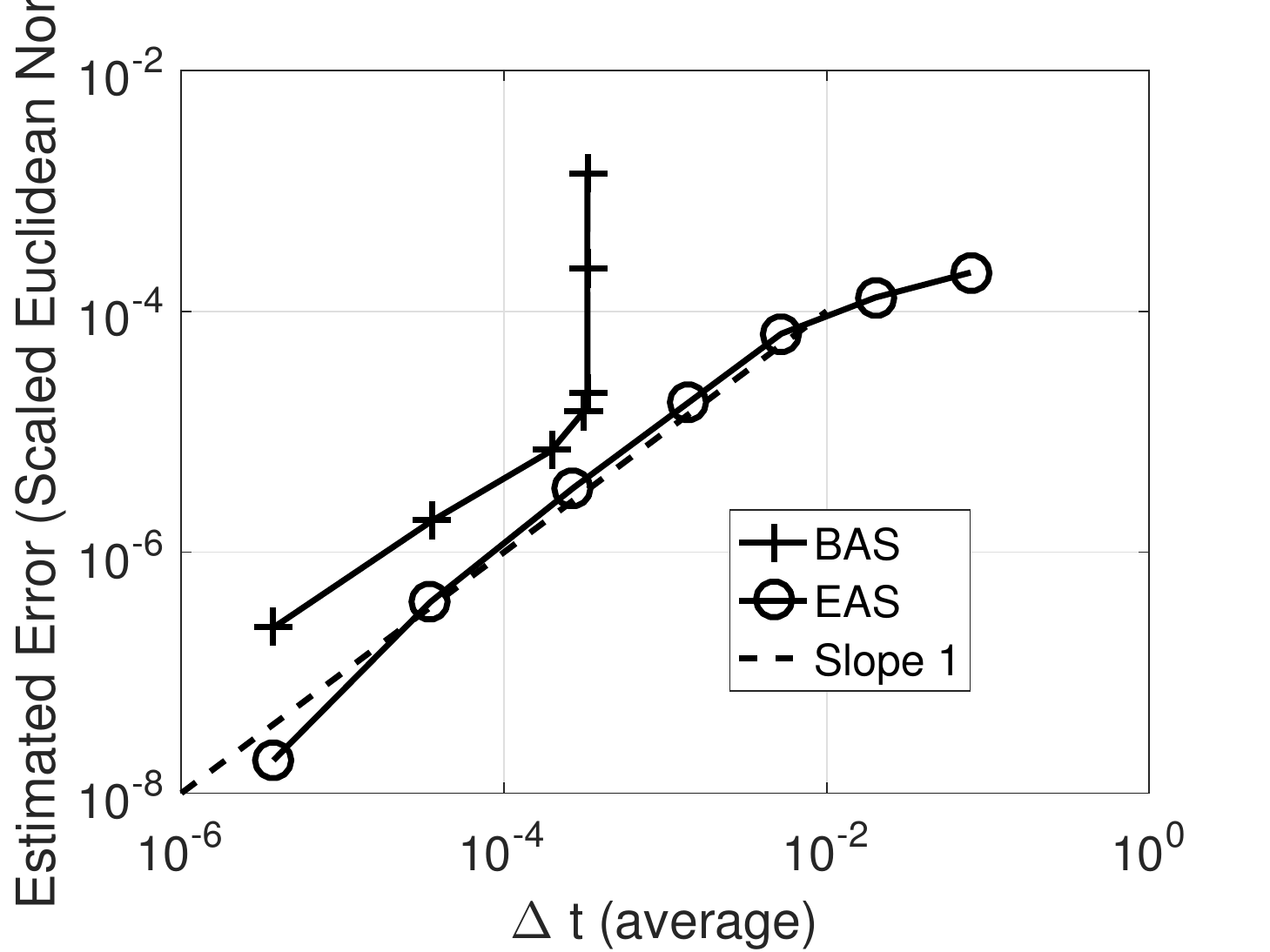}
\end{minipage} 
\begin{minipage}[b]{0.45\linewidth}
c) \\
\includegraphics[width=0.99\columnwidth]{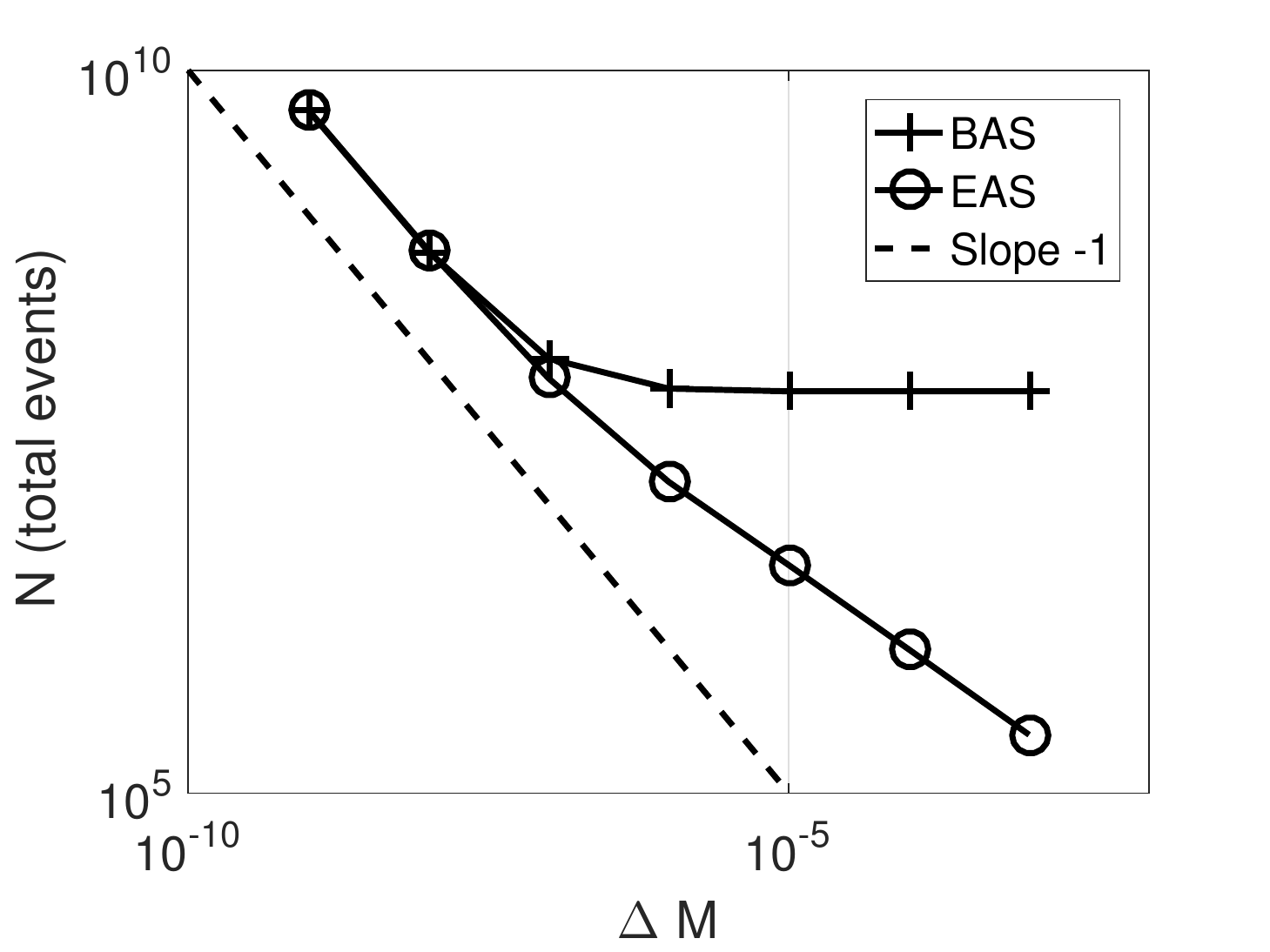}
\end{minipage}
\begin{minipage}[b]{0.45\linewidth}
d) \\
\includegraphics[width=0.99\columnwidth]{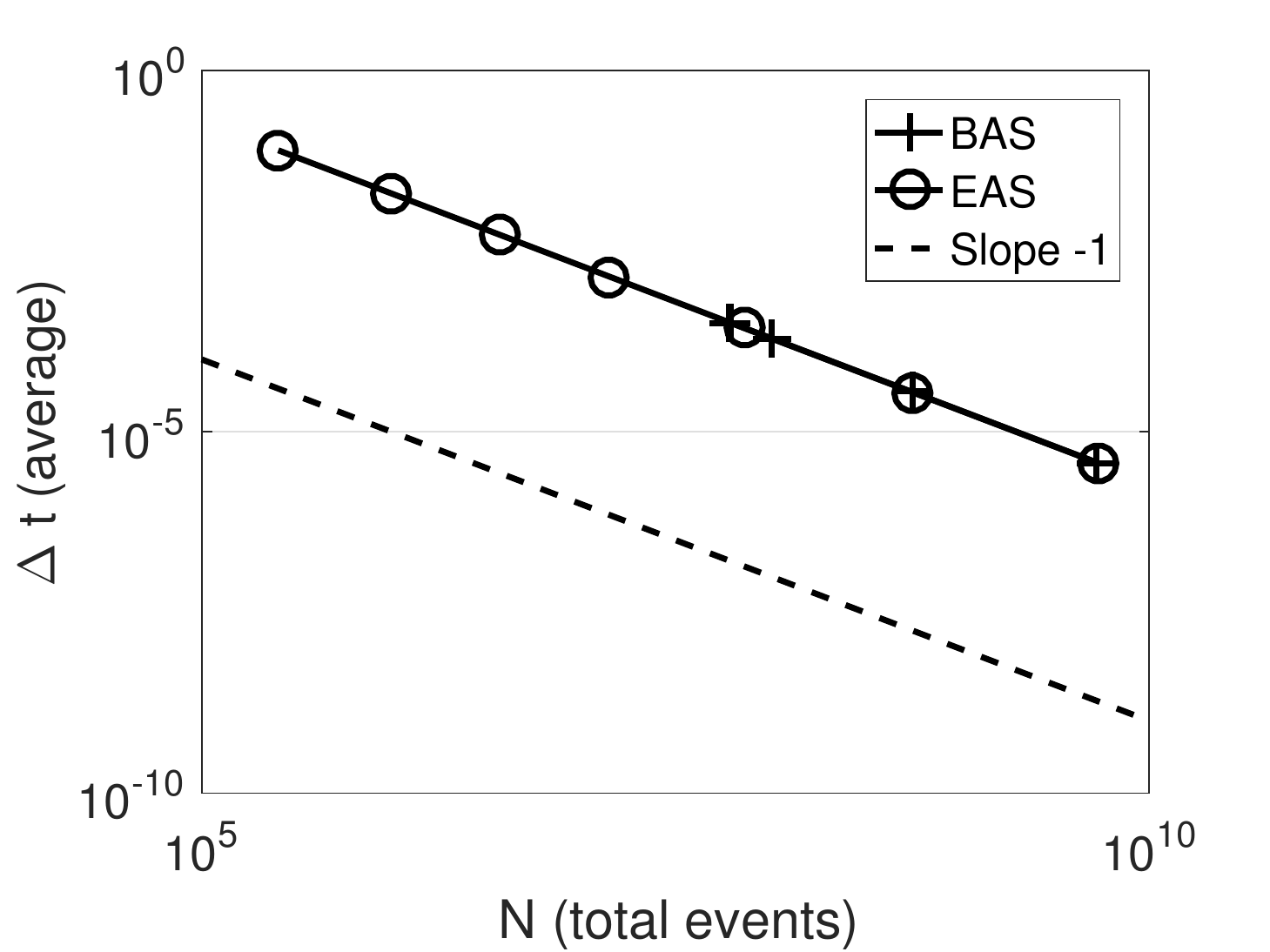}
\end{minipage} 
\caption{Results for the experiment described in \secref{sf_sec}. }
\label{sf_plotsMain}
\end{figure}


\clearpage

\subsection{Fracture System with Langmuir-type Reaction Term}
\label{frac_sec}
The grid is the same as in \secref{fracsec}, and the fracture is defined by the same line of cells. We set the diffusivity to be $D = 100$ on the fracture and $D=0.1$ elsewhere. We specify a simple velocity field. The velocity field was set to be uniformly one in the x-direction and zero in the other directions in the domain, i.e., $\mathbf{v}(\mathbf{x}) = (1,0,0)^T$ , to the right in \figref{fraclang_show}. The initial condition was $c(\mathbf{x})=0$ everywhere except at $\mathbf{x} = (4.95, 9.95)^T$ where $c(\mathbf{x})=1$. \\
We add a spatially dependent Langmuir type reaction term, 
$$
r(c, \mathbf{x}) = - \frac{0.02}{D(\mathbf{x})^2} \frac{c}{1+ c},
$$
In this way the reaction occurs much slower in the fracture than the rest of the domain. Physically this represents the solute species being much less likely to adsorb to the walls of the porous medium, and be lost, within the fracture. The final time is $T=2.4$. \\
In Figures \ref{fraclang_show}, and \ref{fralang_plotsMain} we show the final state of the system, the convergence results, and the parameter relations for the schemes. The layout is the same as in \secref{fracsec} and \secref{sf_sec}. The conclusions are largely the same, even though the system differers from the previous ones due to the addition of a reaction term. 
\begin{figure}[h]
\centering
\begin{minipage}[b]{0.45\linewidth}
a) \\
\includegraphics[width=0.99\columnwidth]{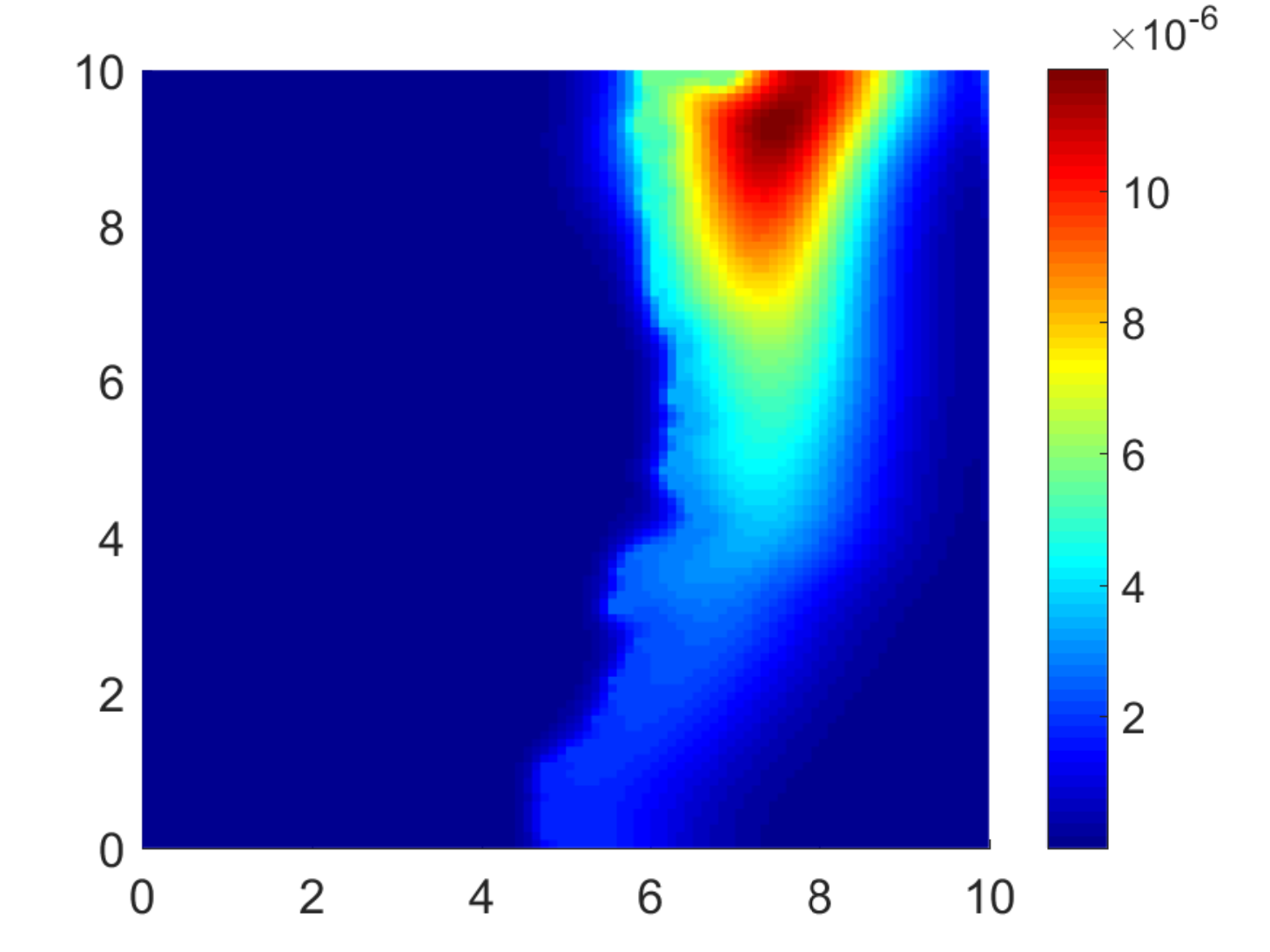}
\end{minipage}
\begin{minipage}[b]{0.45\linewidth}
b) \\
\includegraphics[width=0.99\columnwidth]{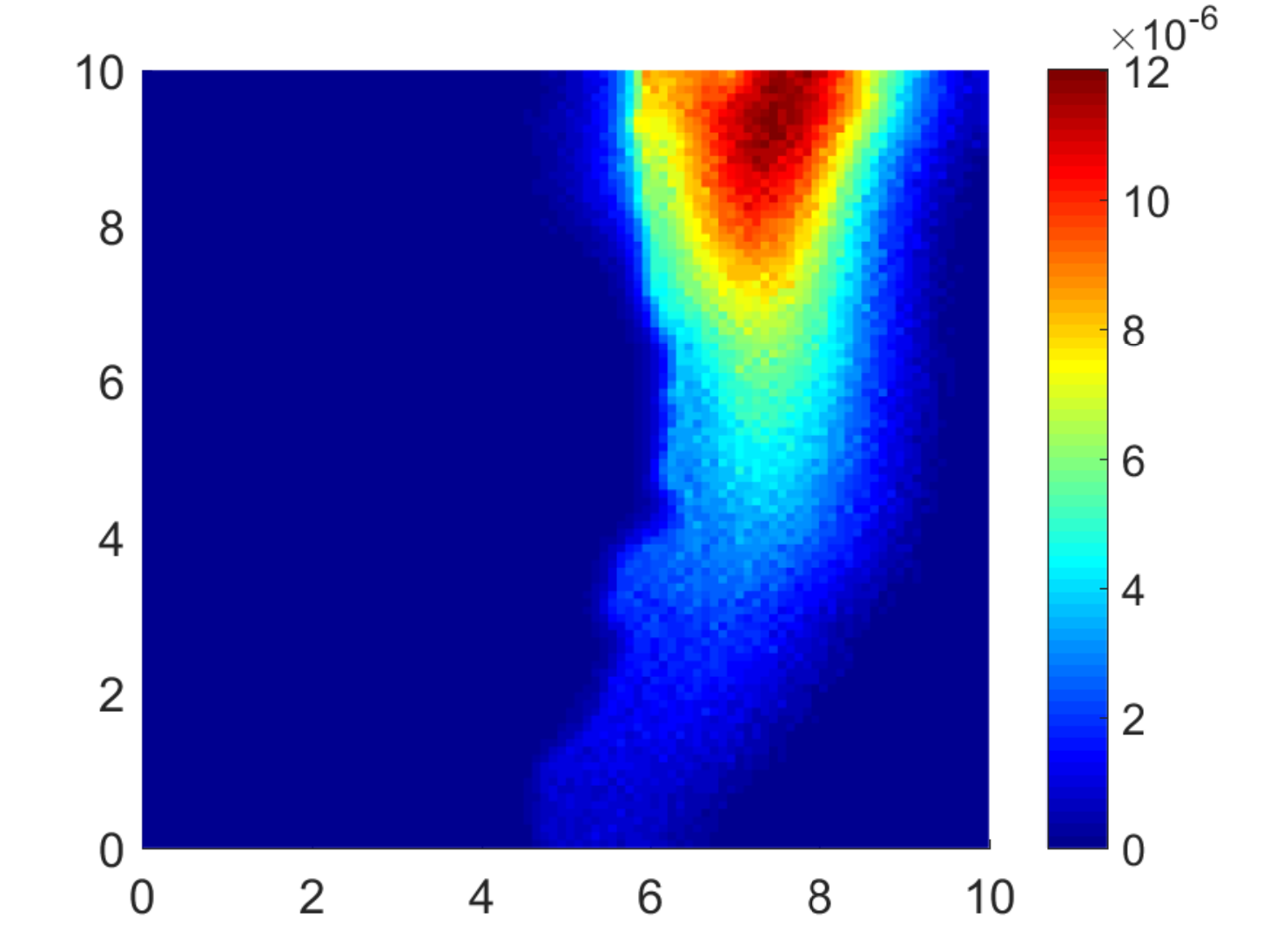}
\end{minipage} 
\begin{minipage}[b]{0.45\linewidth}
c) \\
\includegraphics[width=0.99\columnwidth]{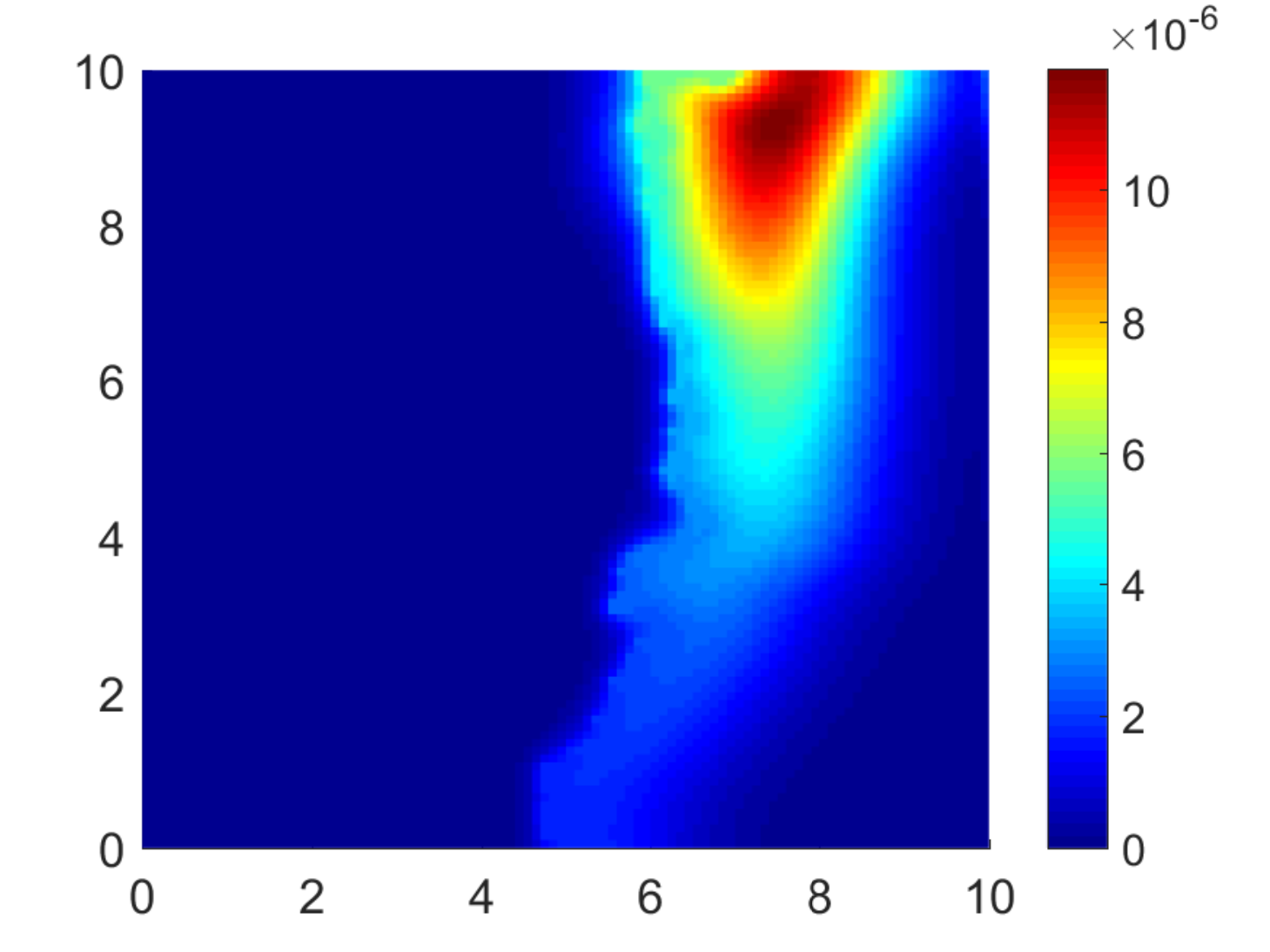}
\end{minipage}
\begin{minipage}[b]{0.45\linewidth}
d) \\
\includegraphics[width=0.99\columnwidth]{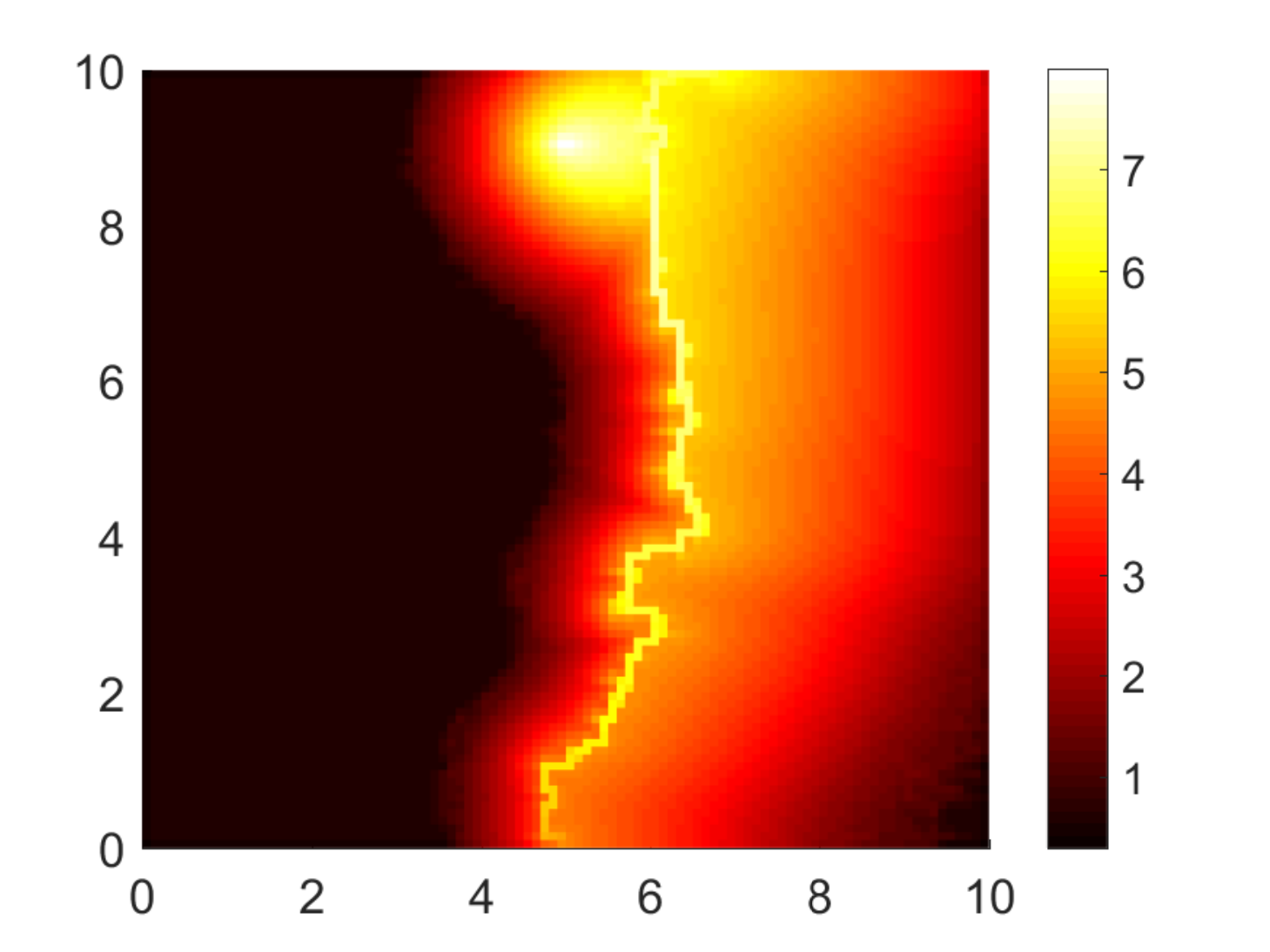}
\end{minipage} 
\label{fraclang_show}
\caption{For the system described in \secref{frac_sec}. a) The comparison solve produced with an exponential integrator. b) Solution produced by EAS with $\Delta M = 10^{-7}$; here $\Delta M$ is too great to allow sufficient accuracy (although positivity is preserved). c) Solution produced by EAS with $\Delta M = 10^{-9}$; this solution is in strong agreement with the comparison solve. d) Shows logarithm of number of events experienced by each cell for the run with EAS and $\Delta M = 10^{-9}$.}
\end{figure}

\begin{figure}[h]
\centering
\begin{minipage}[b]{0.45\linewidth}
a) \\
\includegraphics[width=0.99\columnwidth]{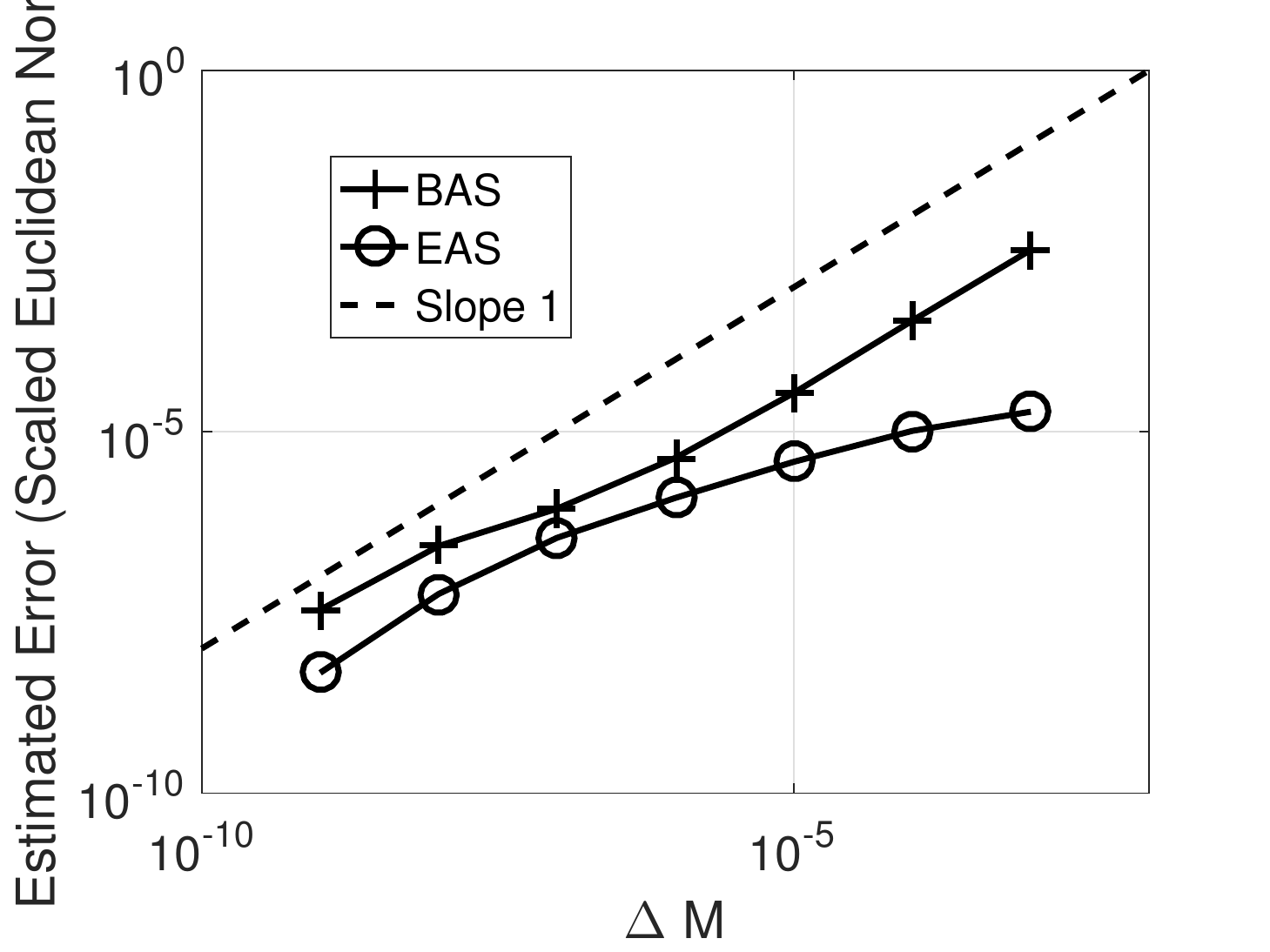}
\end{minipage}
\begin{minipage}[b]{0.45\linewidth}
b) \\
\includegraphics[width=0.99\columnwidth]{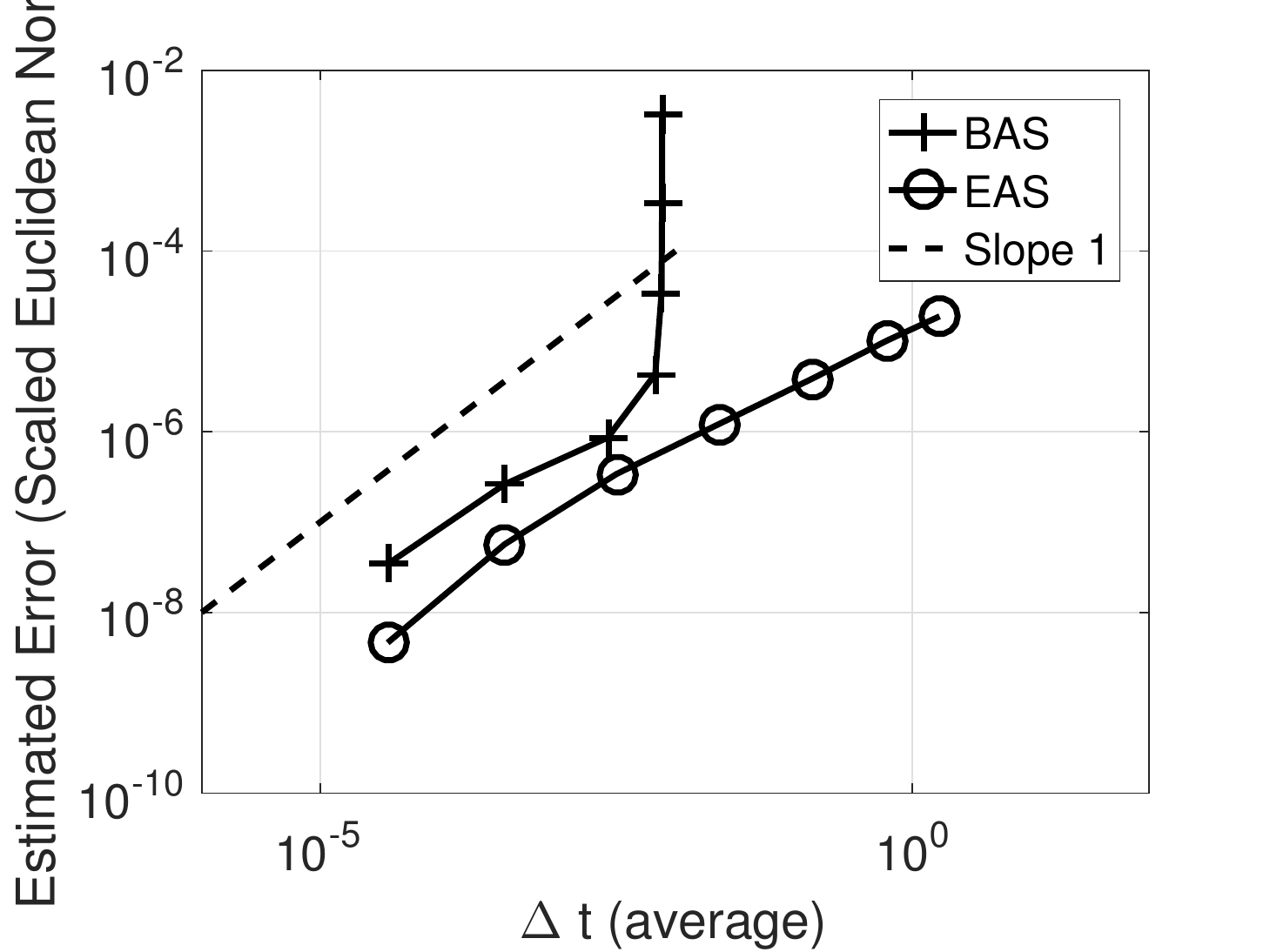}
\end{minipage} 
\begin{minipage}[b]{0.45\linewidth}
c) \\
\includegraphics[width=0.99\columnwidth]{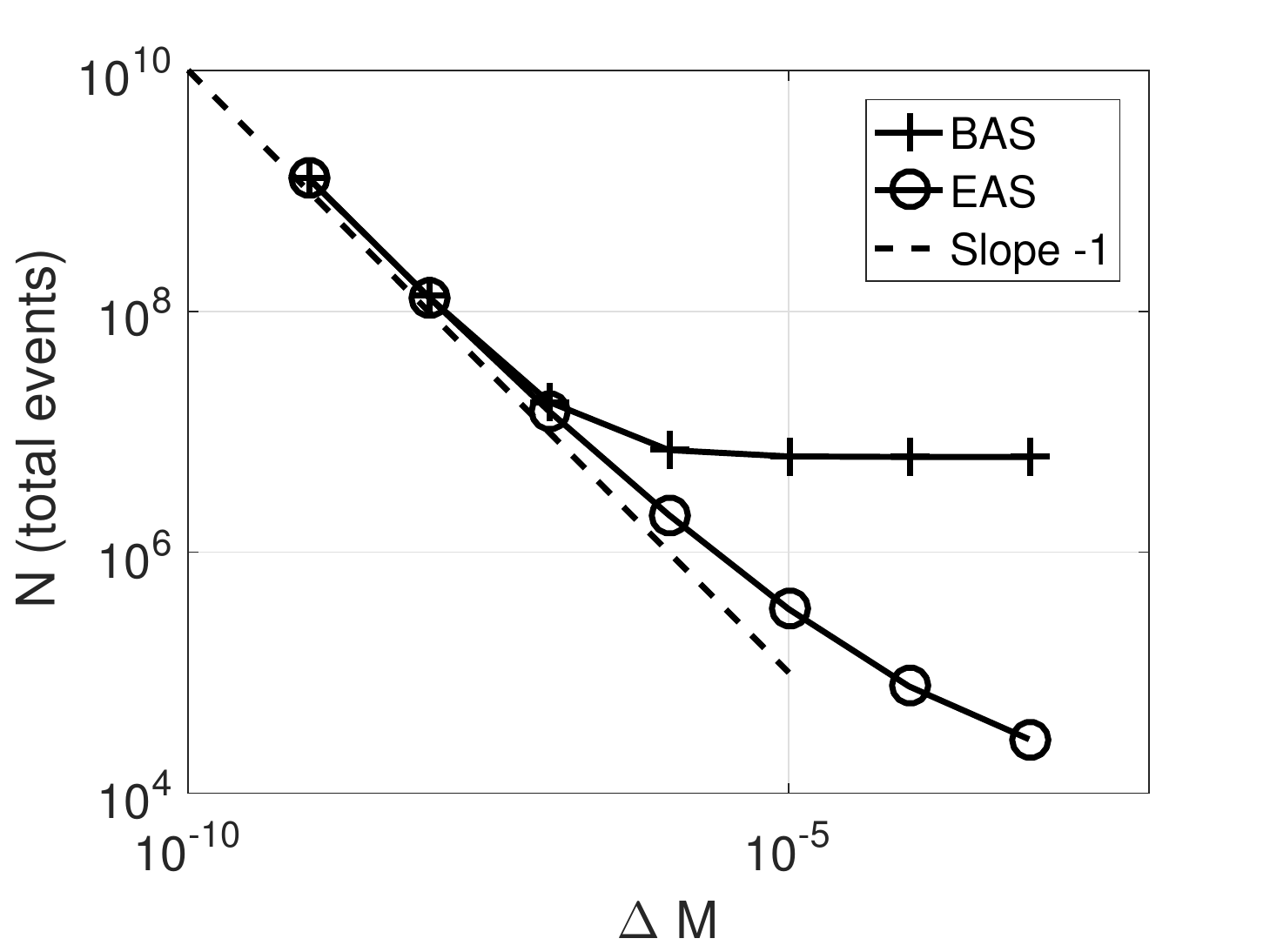}
\end{minipage}
\begin{minipage}[b]{0.45\linewidth}
d) \\
\includegraphics[width=0.99\columnwidth]{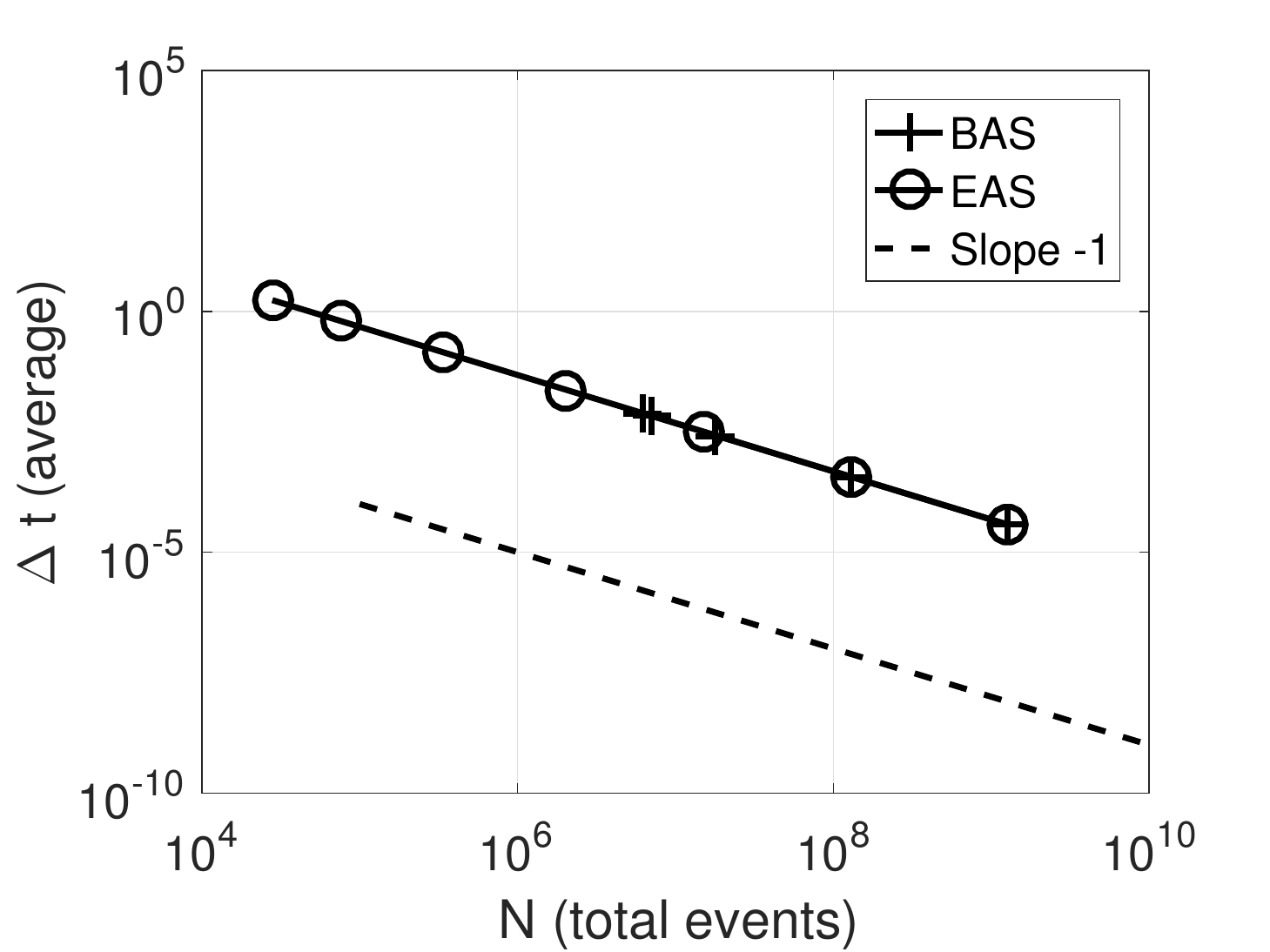}
\end{minipage} 
\caption{Results for the experiment described in \secref{frac_sec}.}
\label{fralang_plotsMain}
\end{figure}

\clearpage

\section{Analysis of EAS}
\label{an_sec_as_1}
Here we perform analysis of the new schemes. First we show theoretically that EAS will never overdraw cells as a result of the scheme's construction, then we present a framework for a convergence analysis of EAS.

It will be seen that a useful property of connection matrices is to re-express the timestep defined by \eqref{utime2} ( i.e., $ \Delta t_{k} = \frac{\Delta M}{|f_{k}|A_k}$). We use the properties of the connection matrix $L_k$ to express the flux across a face $k$ as follows,
$$
|f_{k}|A_k = \frac{|| L_k \mathbf{m}  ||}{\sqrt{2}},
$$
where $|| L_k \mathbf{m} ||$ is the Euclidean norm of $L_k \mathbf{m}  $. The result follows from the fact that the only two nonzero entries of $|| L_k \mathbf{m}  ||$ are $+f_{k}A_k $ and  $-f_{k}A_k$, from \eqref{con_fix_dir}. Then, 
\begin{equation}
\Delta t_{k} = \frac{\sqrt{2} \Delta M}{|| L_k \mathbf{m}  ||}.
\label{dt_eq}
\end{equation}
We make use of \eqref{dt_eq} in the rest of this section.

\subsection{EAS preserves positivity}
\label{pos-sec}

A crucial property of EAS is that it will preserve the positivity of the concentration in cells. BAS by contrast can allow an amount of mass to be transferred between cells such that the concentration in one cell becomes negative - we refer to this phenomenon as overdrawing a cell. See \figref{positivity-fig2} and \figref{positivity-fig3} for observations of this phenomenon in the experiments above.

\begin{figure}[h]
\centering
\begin{minipage}[b]{0.45\linewidth}
a) \\
\includegraphics[width=0.99\columnwidth]{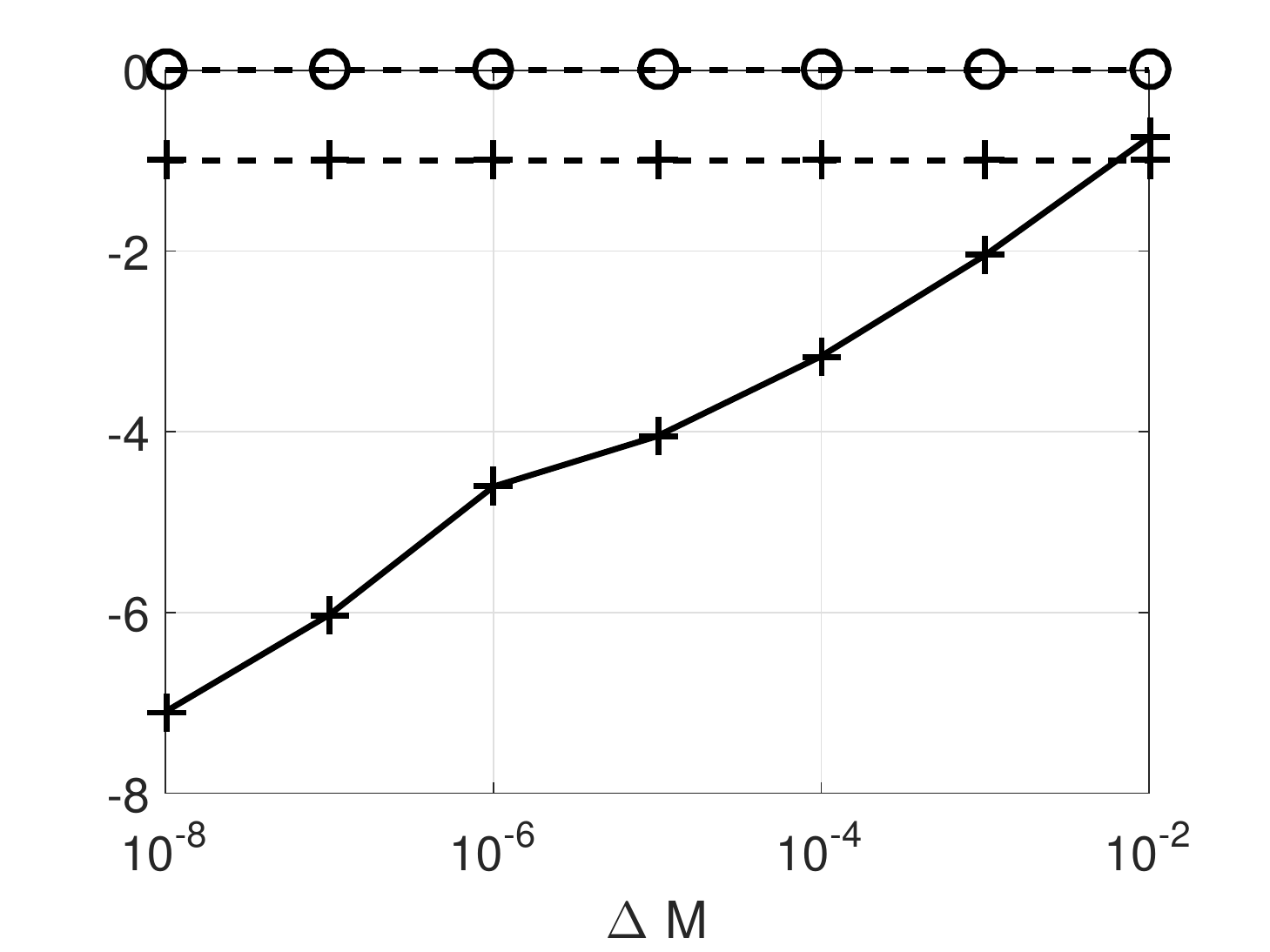}
\end{minipage}
\begin{minipage}[b]{0.45\linewidth}
b) \\
\includegraphics[width=0.99\columnwidth]{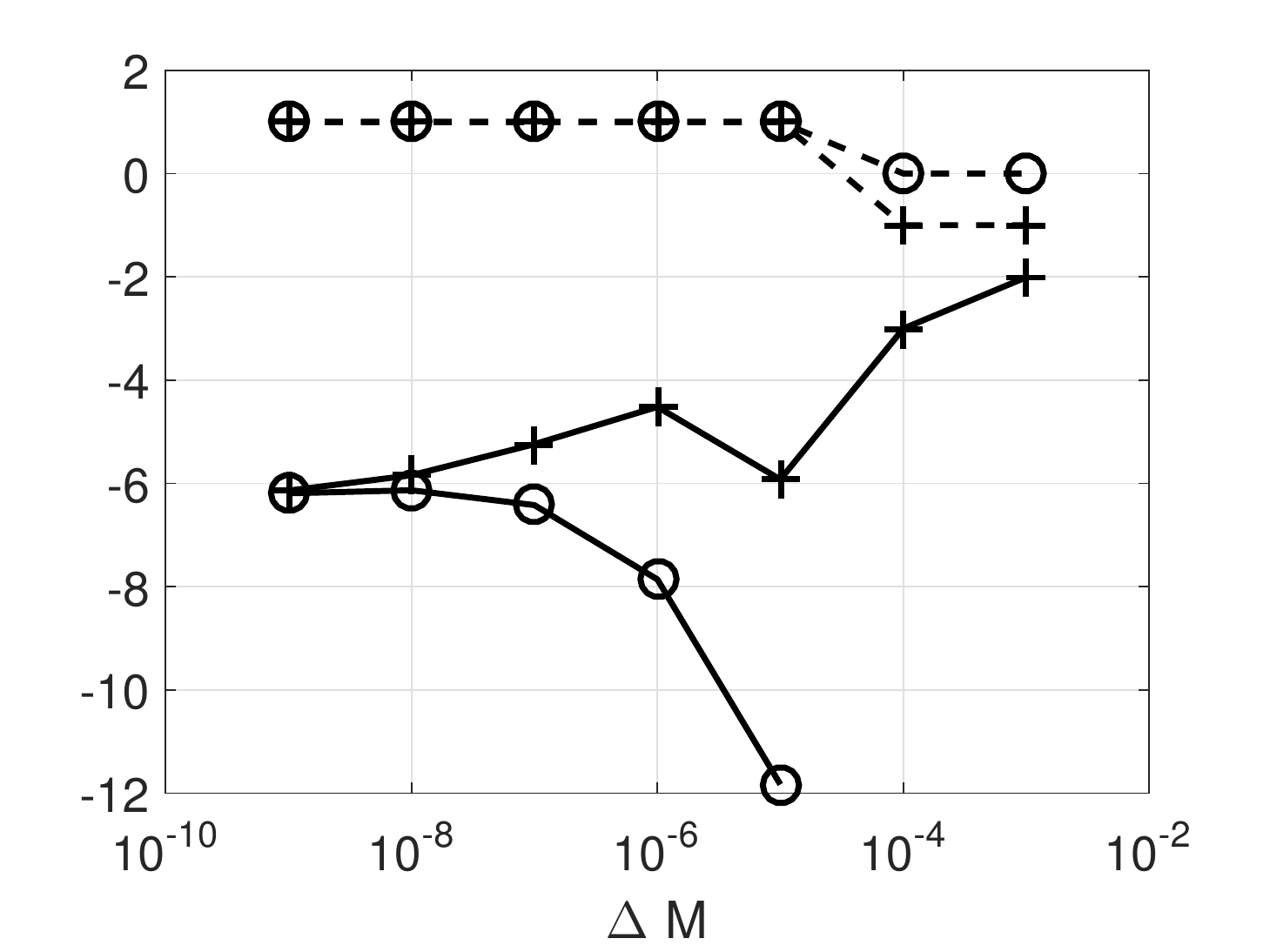}
\end{minipage} 
\begin{minipage}[b]{0.45\linewidth}
c) \\
\includegraphics[width=0.99\columnwidth]{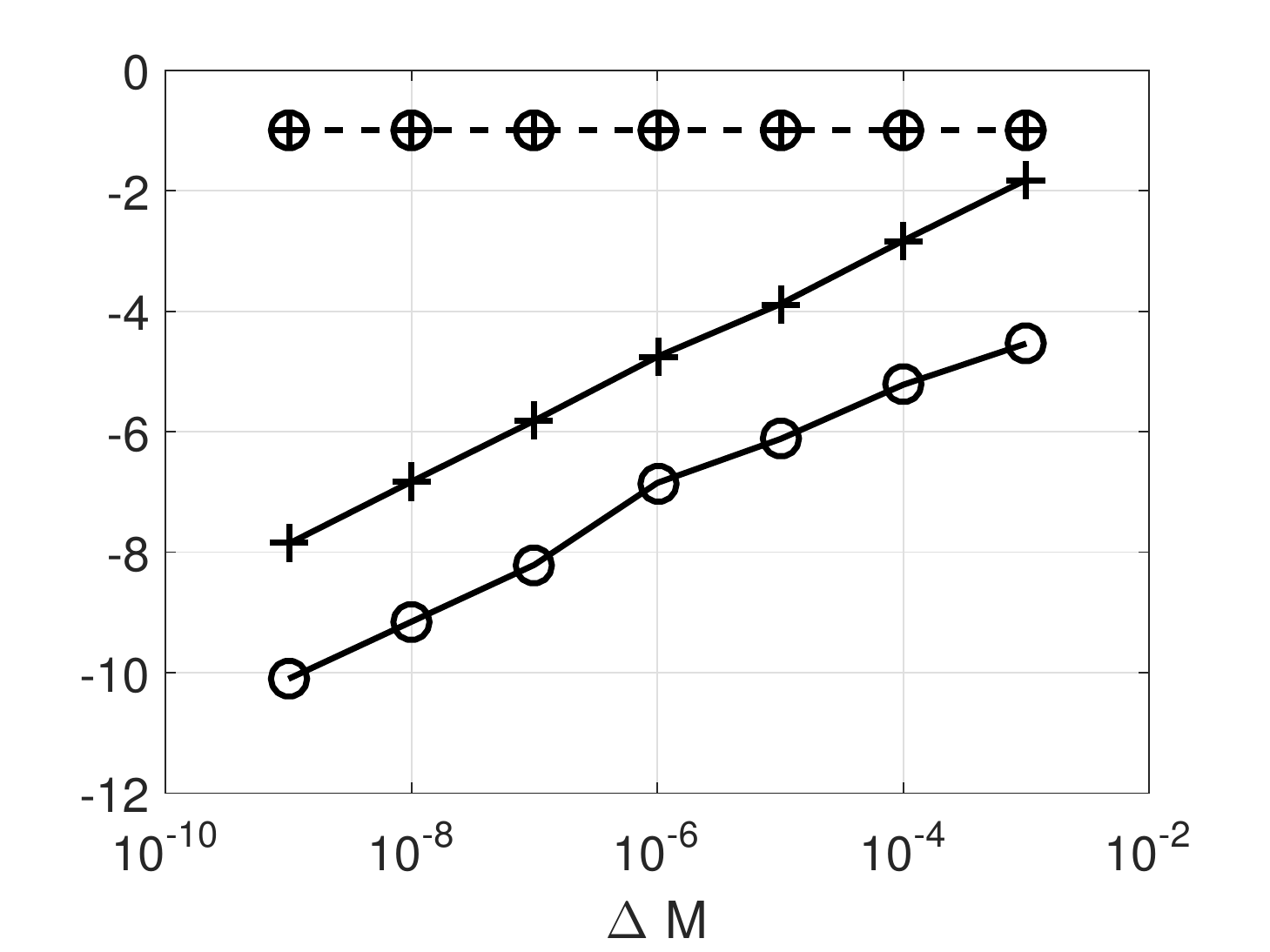}
\end{minipage}
\begin{minipage}[b]{0.45\linewidth}
Legend \\
\includegraphics[width=0.99\columnwidth]{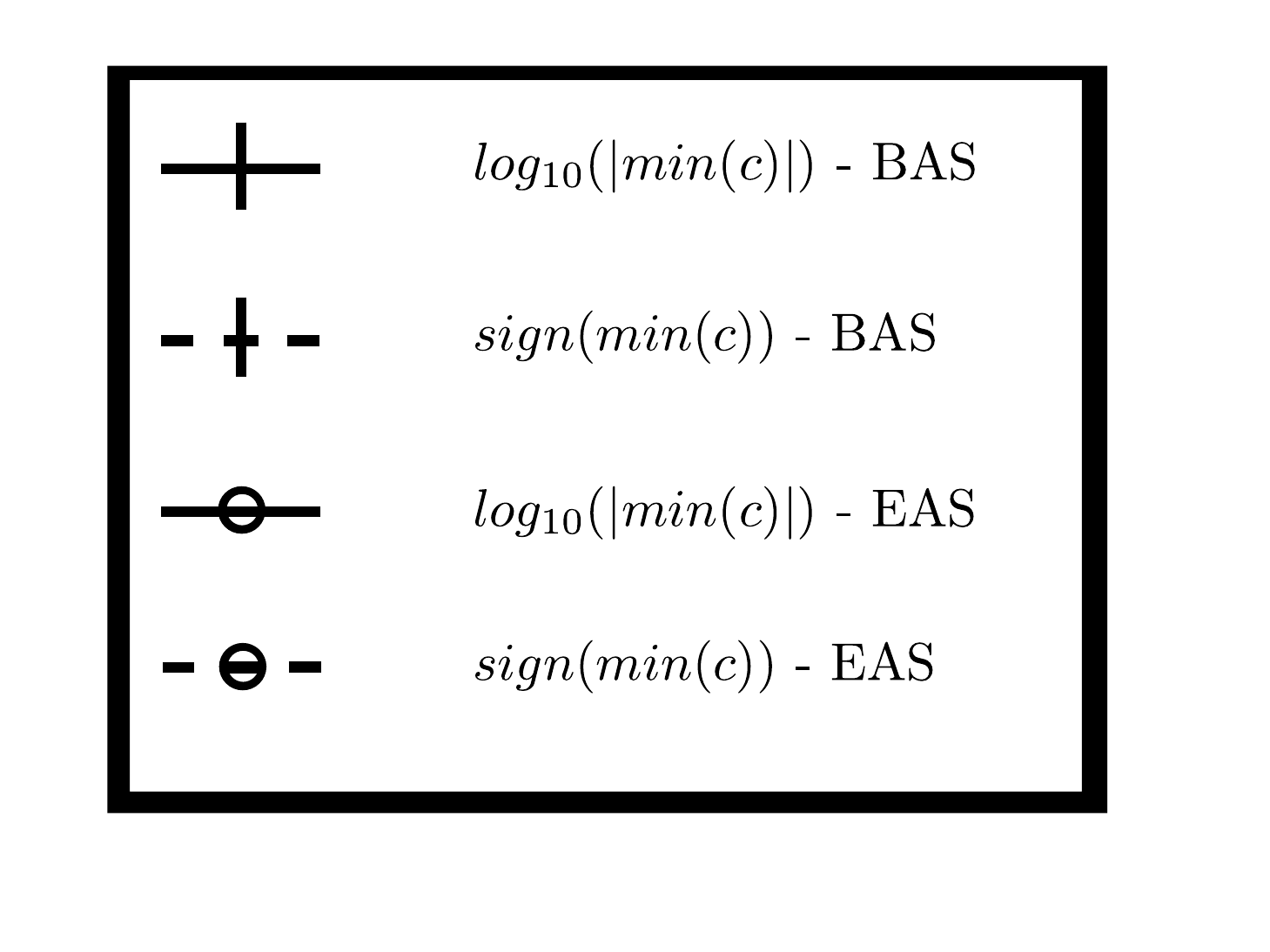}
\end{minipage}
\caption{How EAS preserves positivity and BAS does not. Plot a) corresponds to the experiment in \secref{fracsec}, b) to the experiment in \secref{sf_sec}, and c) to the experiment in \secref{frac_sec}. Note that in a) $min(c) = 0$ for all values for EAS so it cannot be shown on the log-scale plot. a) and b) show that for the first two experiments EAS never reduces $c$ below zero but BAS does. Plot c) shows that EAS does reduce $c$ below zero, due to modifications to allow the reaction term.}
\label{positivity-fig2}
\end{figure}

\begin{figure}[h]
\centering
\begin{minipage}[b]{0.45\linewidth}
a) \\
\includegraphics[width=0.99\columnwidth]{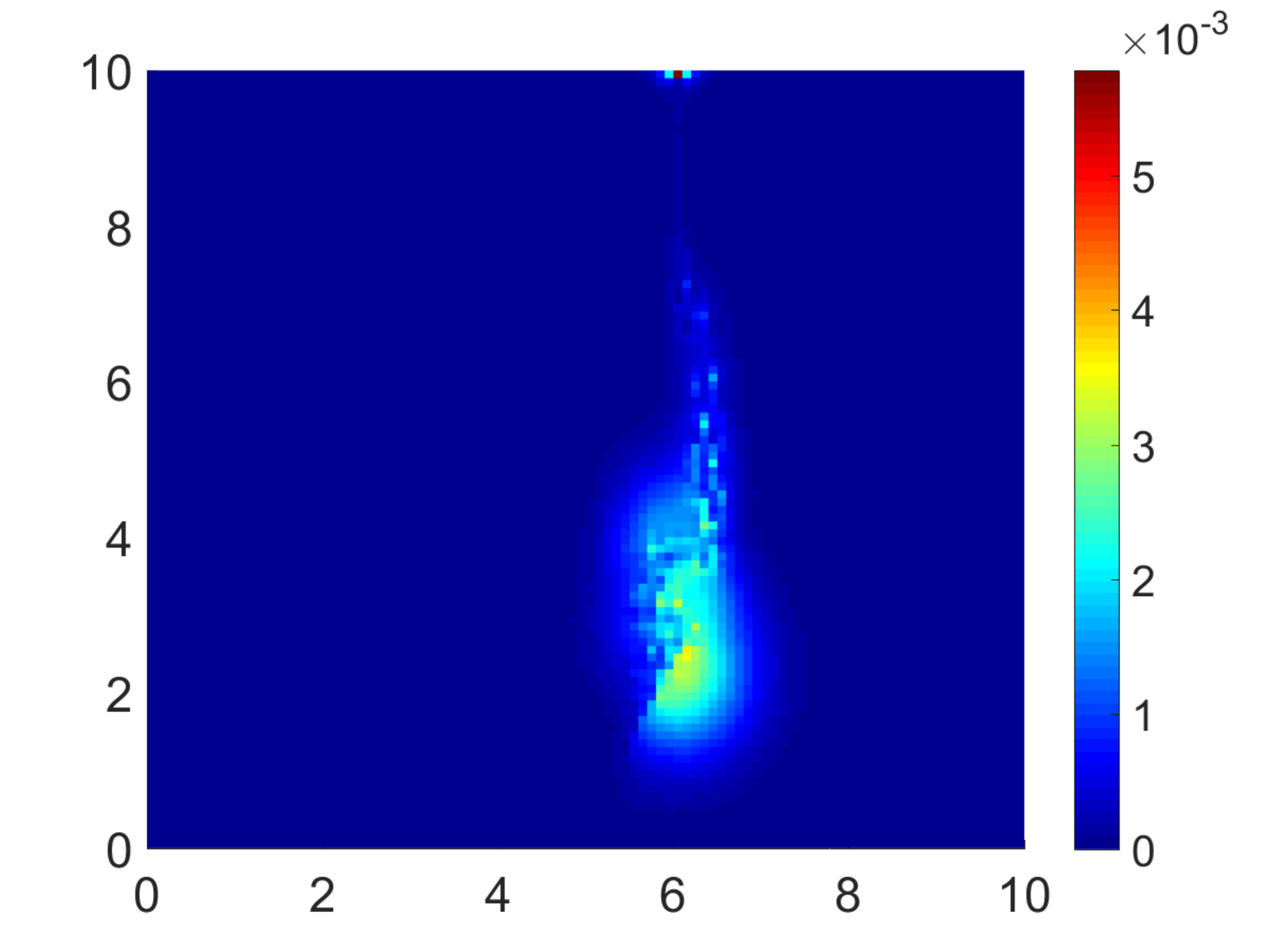}
\end{minipage}
\begin{minipage}[b]{0.45\linewidth}
b)  \\
\includegraphics[width=0.99\columnwidth]{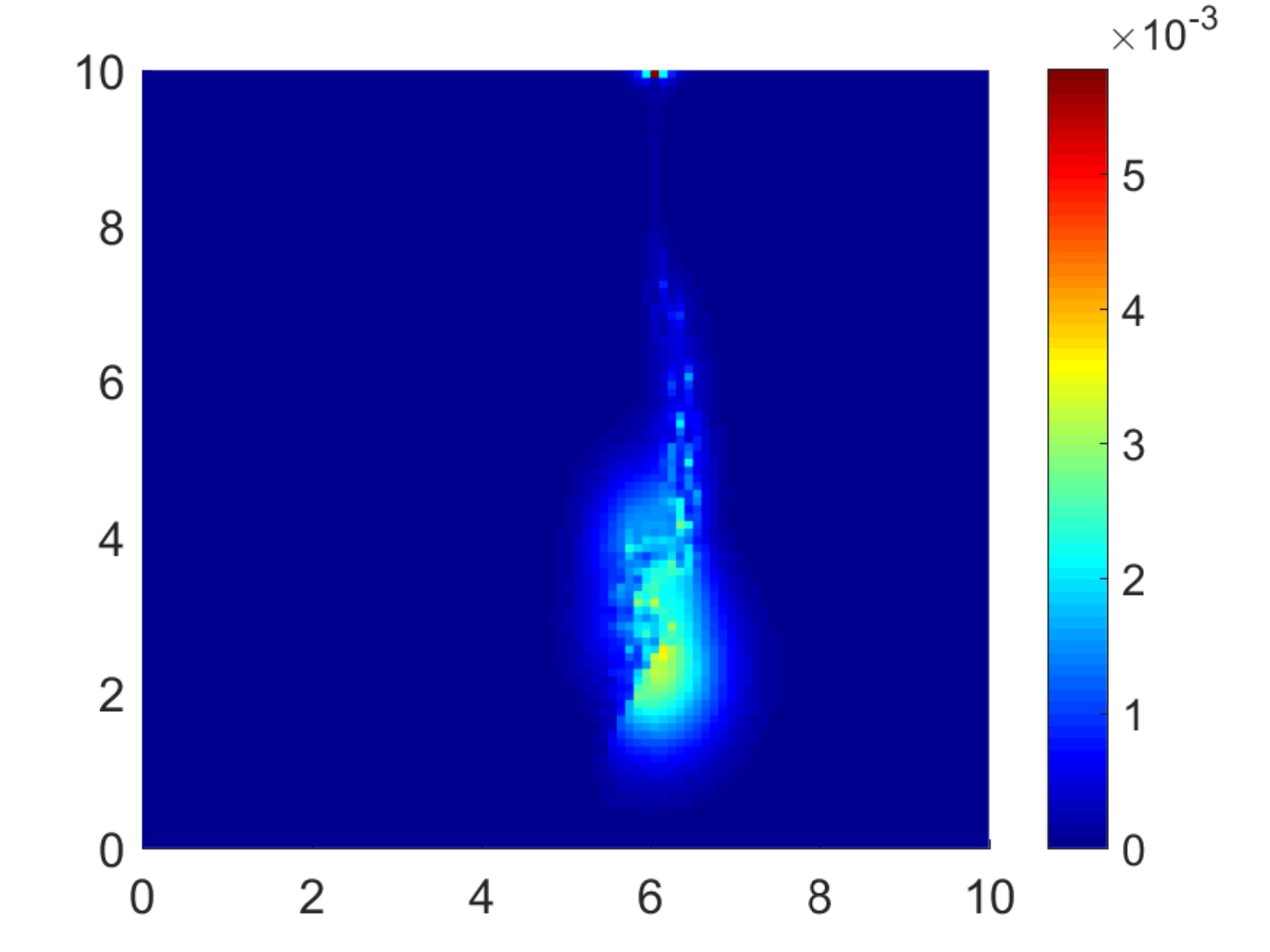}
\end{minipage}
\begin{minipage}[b]{0.45\linewidth}
c) \\
\includegraphics[width=0.99\columnwidth]{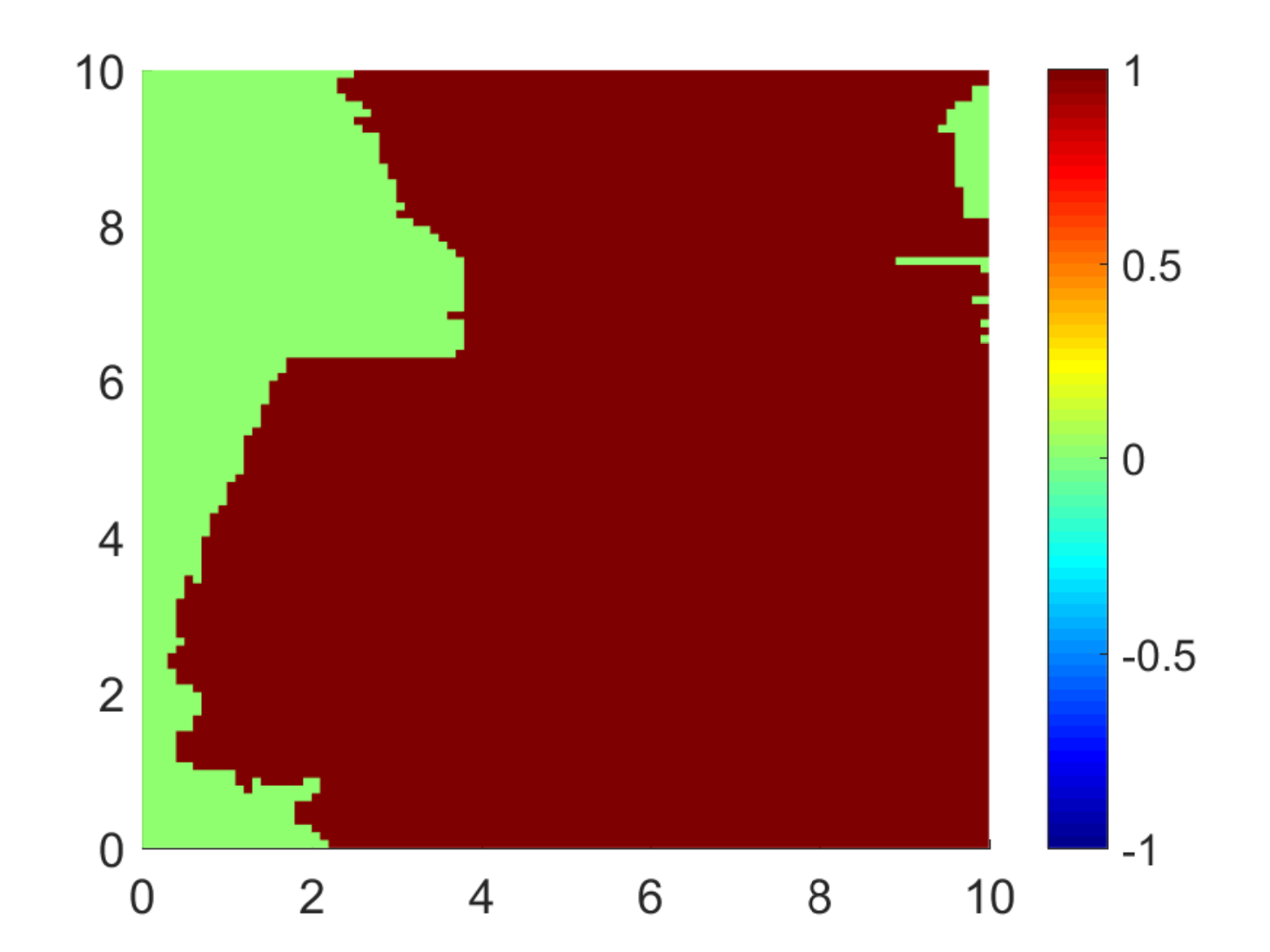}
\end{minipage}
\begin{minipage}[b]{0.45\linewidth}
d) \\
\includegraphics[width=0.99\columnwidth]{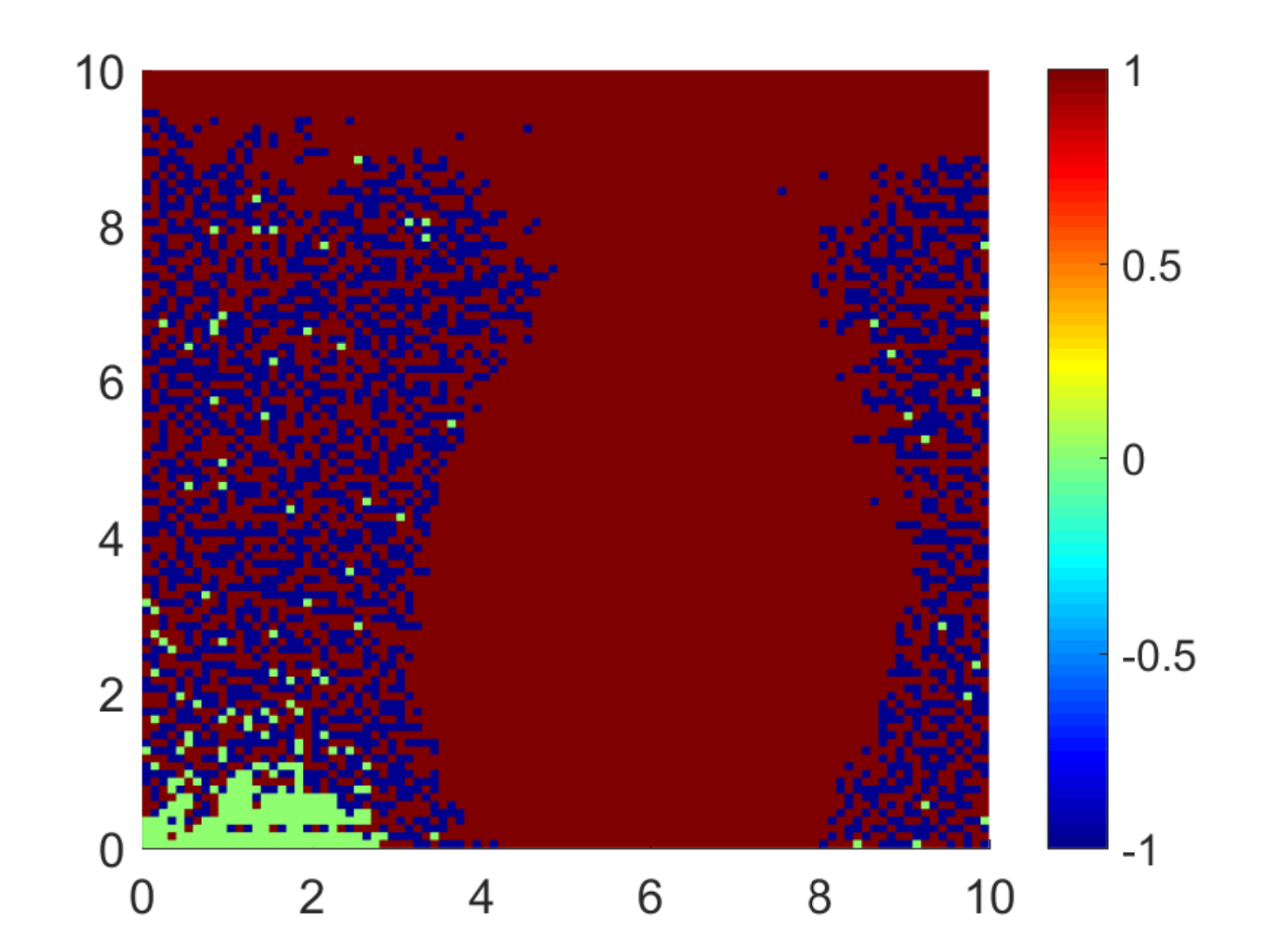}
\end{minipage} 
\caption{Demonstrating the positivity-preserving property of of EAS compared to BAS. a) is produced by EAS for the experiment in \secref{fracsec}. b) is produced by BAS, also for the experiment in \secref{fracsec}. c) shows the sign of the solution in a), and d) shows the sign of the solution in d). While  a) and b) show two similar-looking solves; in plots c) and d)  we observed that BAS allows the solution to become negative while EAS does not.}
\label{positivity-fig3}
\end{figure}

A theoretical proof that EAS will not overdraw cells is as follows. We can show that if the $j_1$, $j_2$ entries in $\mathbf{x}$ (i.e. $x_{j_1} $ and $x_{j_2} $) are both non-negative, then the $j_1$, $j_2$ entries in $e^{sL_k}\mathbf{x}$ are also non-negative. This is a corollary to Lemma \ref{phi_con_mats}.
\begin{corollary}
With the same assumptions as in Lemma \ref{phi_con_mats}, the $j_1$, $j_2$ entries in $e^{sL_k}\mathbf{x}$ are given by $x_{j_1}+ \Delta \hat{M}_{k, j_1, j_2} (s)$ and $x_{j_2}-\Delta M_e(s)$ respectively. Both of these are non-negative if $x_{j_1} $ and $x_{j_2}$ are non-negative.
\label{exp_mat_dec}
\end{corollary}
\begin{proof}
The first claim follows simply from \eqref{phi_con_2} and the form of $\hat{\mathbf{z}}_k$. For the second claim, consider $\Delta M_e(s)$ re-written as
$$
\Delta M_e(s) =(1-e^{-s(a_k+b_k)})  \frac{(b_k x_{j_2} -a_k x_{j_1})}{a_k+b_k}.
$$
The $(1-e^{-s(a_k+b_k)}) $ part is a monotonically increasing function, from $0$ when $s=0$ to $1$ as $s \rightarrow \infty$, so it is in $[0,1)$. The other part of $\Delta M_e(s)$ is $\frac{(b_k x_{j_2} -a_k x_{j_1})}{a_k+b_k}$ and can be either positive or negative. We can consider each case separately. \\
If $\frac{(b_k x_{j_2} -a_k x_{j_1})}{a+b} \geq 0$, then
$$
0 \leq \Delta M_e(s) < \frac{(b_k x_{j_2} -a_k x_{j_1})}{a_k+b_k},
$$
and so,
$$
x_{j_1} \leq x_{j_1} +  \Delta M_e(s) < x_{j_1} +  \frac{(b_k x_{j_2} -a_k x_{j_1})}{a_k+b_k},
$$
and
$$
x_{j_2} \geq x_{j_2} -  \Delta M_e(s) > x_{j_2} -  \frac{(b_k x_{j_2} -a_k x_{j_1})}{a_k+b_k} = \frac{a_k(x_{j_1}+x_{j_2})}{a_k+b_k},
$$
so that both $x_{j_1} +  \Delta M_e(s)$ and $ x_{j_2} -  \Delta M_e(s)$ are non-negative. \\
 If $\frac{(b_k x_{j_2} -a_k x_{j_1})}{a_k+b_k} \leq 0$ then,
$$
0 \geq \Delta M_e(s) > \frac{(b_k x_{j_2} -a_k x_{j_1})}{a_k+b_k},
$$
leading to
$$
x_{j_1} \geq x_{j_1} + \Delta M_e(s) > x_{j_1} + \frac{(b_k x_{j_2} -a_k x_{j_1})}{a_k+b_k} = \frac{b_k( x_{j_1} + x_{j_2} ) }{a_k+b_k},
$$
and
$$
x_{j_2} \leq x_{j_2} - \Delta M_e(s) < x_{j_2} - \frac{(b_k x_{j_2} -a_k x_{j_1})}{a_k+b_k}.
$$
Thus we have that $x_{j_1} +  \Delta M_e(s)$ and $ x_{j_2} -  \Delta M_e(s)$ are non-negative in either case.
\end{proof}

\clearpage

\subsection{Assumptions of Parameter Relations}
\label{asum-sec}
Identifying relationships between the parameters such as $\Delta M$, $N$, and individual event timesteps $\delta t_i$ would be essential for a full analysis of the schemes presented here. Some relationships are heavily implied by our numerical results in \secref{numerical}; others suggest themselves from the form of the scheme, but there are subtleties to consider. We first state three assumptions that we base on the form of the schemes, which we use in the \secref{an_EAS} in a sketch proof of convergence for EAS, and then discuss these assumptions. 
\begin{assumption}
There exists a $d_1 > 0$ such that
$
\delta t_n = O \left(\Delta M^{d_1} \right), \mbox{  } n=1, \ldots, N.
$
\label{as_dm_dt}
\end{assumption}
That is, for an event number $n$, if $\Delta M$ is small enough that event $n$ occurs (i.e., that $n<N$, see Assumption \ref{as_N_dm}), the event timestep decreases as $\Delta M$ decreases. This is the observed in the numerical results (plot d) in all of Figures \ref{frac-flow-plotsMain}, \ref{sf_plotsMain} and \ref{fralang_plotsMain}), and \eqref{dt_eq}, which implies a proportionality between the timestep $\Delta t$ and $\Delta M$. However, we observe that the denominator on the right hand side of \eqref{dt_eq} may have some dependence on $\Delta M$ since the ordering of events may depend on $\Delta M$. This means that $||L_k \mathbf{m}||$ is different for different $\Delta M$ values for given $n$. Ideally we would like to identify a $C$ such that
$$
||L_k \mathbf{m}|| \geq 1/C,
$$
which is the same as claiming that, for any event $n$ and $\Delta M$ value, the flux across the face $k$ \emph{chosen for an event} is bounded below by some constant $1/C$. A face $k$ is chosen for an event due to a combination of low $t_k$ and high flux; this might make the existence of such a bound seem reasonable, but this is not conclusive for all cases. Our numerical results that support the assumption indicate that the exponent has a value around $d_1 = 1$.
\begin{assumption}
The number of events $N$ increases as $\Delta M$ decreases, and there exists a $d_2 >0 $ such that
$
N = O \left(\Delta M^{-d_2} \right).
$
\label{as_N_dm}
\end{assumption}
As well as being implied strongly by our numerical results (plot c) in all of Figures  \ref{frac-flow-plotsMain}, \ref{sf_plotsMain} and \ref{fralang_plotsMain}, this assumption should follow from the construction of the algorithms as long as the initial data is such that there is some activity in the domain (i.e., a nonzero flux). 

Bringing together Assumptions \ref{as_dm_dt} and \ref{as_N_dm} gives a third parameter relation.
\begin{remark}
Given Assumptions \ref{as_dm_dt} and \ref{as_N_dm}, as $N$ increases, the event timestep decreases; there exists  $d_3 >0$ such that
$
\delta t_n = O \left(N^{-d_3} \right), \mbox{  } n=1, \ldots, N.
$
\label{as_dt_N}
\end{remark}
This is supported by combining our numerical results (plot c and d) in all of Figures \ref{frac-flow-plotsMain}, \ref{sf_plotsMain} and \ref{fralang_plotsMain}, for the average $\Delta t$, with $d_3 = 1$.

\subsection{Discusion and Steps Towards Convergence}
\label{an_EAS}
We consider the linear ODE system \eqref{fv_disc}, after dividing each row by $V_k$ to produce an equation for the mass values, i.e., $\frac{d  \mathbf{m}}{dt}=L  \mathbf{m}$. The exact solution is,
\begin{equation}
\mathbf{m}(T) = \exp \left( T\sum_{k=1}^{K}{L_k}  \right) \mathbf{m}(0),
\label{lin_ex_sol}
\end{equation}
where we have expressed $ TL = T\sum_{k=1}^{K}{L_k}$, as a sum of the connection matrices. The approximation produced by EAS after $n$ events is
\begin{equation}
\mathbf{m}_n = \prod_{i=1}^{n}{\exp(\delta t_{i} L_{k_i}  ) }\mathbf{m}(0)  = \exp{(Z_n)}\mathbf{m}(0),
\label{EAS_after_n}
\end{equation}
where $L_{k_i} \in \mathcal{L}$ is the connection matrix chosen for the $i$th event, and $Z_n$ is to be determined. The iterative formula for $\mathbf{m}_n$ is 
\begin{equation}
\mathbf{m}_{n+1} = \exp{ (\delta t_{n+1} L_{k_{n+1}} ) }\mathbf{m}_{n}  = \exp{ (\delta t_{n+1} L_{k_{n+1}} ) }\exp{(Z_n)}\mathbf{m}(0).
\label{EAS_it}
\end{equation}
We can make use of the Baker-Campell-Hausdorff (BCH) formula (see for example, \cite{HS_lie_book}), which states that, for operators $A$, $B$, 
\begin{equation}
\exp{(A)}\exp{(B)} = \exp{(C)},
\label{expabc}
\end{equation}
with 
\begin{equation}
C = A + B + \frac{1}{2}[A,B] + \frac{1}{12}[A,[A,B]] - \frac{1}{12}[B,[B,A]] + \ldots,
\label{Z_BCH}
\end{equation}
where the Lie bracket $[\cdot]$ is defined as $[A,B] = AB - BA$, for $A,B \in \mathbb{R}^{J \times J}$. We have an iterative formula for $Z_n$  from \eqref{EAS_it}, by taking $A= \delta t_{n+1} L_{k_{n+1}}$ and $B=Z_n$ in \eqref{Z_BCH},
\begin{equation}
\begin{split}
Z_{n+1} &= \delta t_{n+1} L_{k_{n+1}} + Z_n + \frac{1}{2}[\delta t_{n+1} L_{k_{n+1}},Z_n] \\
              &+ \frac{1}{12}[\delta t_{n+1} L_{k_{n+1}},[\delta t_{n+1} L_{k_{n+1}},Z_n]] - \frac{1}{12}[Z_n,[Z_n,\delta t_{n+1} L_{k_{n+1}}]] + \ldots.
\end{split}
\label{Zn_BCH}
\end{equation}
An alternative expression of \eqref{Z_BCH} is the Goldberg series \cite{goldberg_BCH}, see also for example \cite{conv_goldberg}. The Goldberg series is a double sum of words made of the operators $A$ and $B$; a word here means a simple multiplicative term, for example $A$, $B$, $ABA$, $AABB$ are all examples of words made from $A$ and $B$. A word of length $i$ is made of $i$ instances of the operators, for example $A$ and $B$ are length one, $ABA$ of length three and $AABB$ length four. There are $2^i$ words of length $i$ that can be made from operators $A$ and $B$ ($3^i$ if there were three operators, and so on). For the purposes of writing a sum over the words, let $W(j,i,A,B)$ be the $j$th word of length $i$ ($j =1, \ldots, 2^i$) made from $A$ and $B$. For example, $W(j,1,A,B)$ could be $A$ or $B$; $ABA$ would be one of the $W(j,3,A,B)$; $AABB$ would be one of the $W(j,4,A,B)$, and so on. Then, Goldberg's exponential series for $C$ in \eqref{expabc} is 
\begin{equation}
C = A + B + \sum_{i=2}^{\infty}{ \sum_{j=1 }^{2^i}{  g(j,i,X,Y) W(j,iX,Y) } }.
\label{goldberg}
\end{equation}
The Goldberg coefficients $ g(j,i,X,Y)$ corresponding to each word are rational numbers, and listings and discussions of the calculations of these can be found in \cite{list_goldberg}. \\
An advantage of EAS is that it the scheme can be written in exponential form \eqref{EAS_after_n}, which lends itself to analysis using the BCH. In the spirit of analysis of symplectic operator splitting schemes \cite{yoshida_symp}, we can attempt to prove convergence of $\mathbf{m}_N$ to $\mathbf{m}(T)$, by proving convergence of $Z_N$ to $ TL = T\sum_{k=1}^{K}{L_k}$, where $N$ is the number of events in the EAS solve (i.e., the number of events after which the scheme has brought the individual time on every face to $T$). We demonstrate here how such an argument may proceed. \\
It is useful to express $Z_n$ from \eqref{EAS_after_n} in a modified form of \eqref{goldberg}. We use words made from the event timesteps instead of operators. Let $\hat{w}(j,i,n)$ be the $j$th word of length $i$, made from elements of the set $\{ \delta t_1, \ldots, \delta t_n \}$. Because the timesteps are scalars they commute, unlike the operator words that make up the Goldburg series. For example, $ABA \neq BAA$, but $\delta t_1 \delta t_2\delta t_1$ and $\delta t_2 \delta t_1\delta t_1$ are both equal to $\delta t_1^2 \delta t_2$. Because of this the number of possible words $\hat{w}(j,i,n)$ is given by the multiset coefficient $\binom{n+i-1}{i}$. We write 
\begin{equation}
Z_n =  \sum_{i=1}^{\infty}{ \sum_{j=1 }^{\binom{n+i-1}{i}}{  \hat{g}(j,i,n) \hat{w}(j,i,n) } },
\label{dtwords_goldburg}
\end{equation}
where the modified coefficients $ \hat{g}(j,i,n) $ are not rational numbers but linear combinations of operator words made from elements of the set $\{ L_{k_1}, \ldots, L_{k_n} \}$. An example is helpful. Consider advancing from $n=1$ to $n=2$. Clearly $Z_1 = \delta t_1 L_{k_1}$, and expanding \eqref{Zn_BCH} gives us
\begin{align*}
Z_2 &= \delta t_1 L_{k_1} + \delta t_2 L_{k_2}  \\
&       + \frac{1}{2} \delta t_1 \delta t_2 L_{k_1} L_{k_2}  - \frac{1}{2}  \delta t_1 \delta t_2 L_{k_2} L_{k_1} \\
&       + \frac{1}{12}\delta t_1^2 \delta t_2 L_{k_1} L_{k_1} L_{k_2}    - \frac{1}{6} \delta t_1^2 \delta t_2 L_{k_1} L_{k_2} L_{k_1} + \frac{1}{12} \delta t_1^2 \delta t_2 L_{k_2} L_{k_1} L_{k_1}  \\
&        - \frac{1}{12 } \delta t_1 \delta t_2^2 L_{k_2} L_{k_2} L_{k_1}  + \frac{1}{6}\delta t_1 \delta t_2^2 L_{k_2} L_{k_1} L_{k_2}  - \frac{1}{12} \delta t_1 \delta t_2^2 L_{k_1} L_{k_2} L_{k_2}   \\
&       + \ldots
\end{align*}
Collecting the timestep words this is
\begin{align*}
Z_2 &= \delta t_1 L_{k_1} + \delta t_2 L_{k_2}  \\
&       +  \delta t_1 \delta t_2 \left(\frac{1}{2} L_{k_1} L_{k_2}  - \frac{1}{2}  L_{k_2} L_{k_1} \right) \\
&       + \delta t_1^2 \delta t_2 \left(  \frac{1}{12} L_{k_1} L_{k_1} L_{k_2}    - \frac{1}{6}  L_{k_1} L_{k_2} L_{k_1} + \frac{1}{12}  L_{k_2} L_{k_1} L_{k_1}  \right) \\
& +    \delta t_1 \delta t_2^2   \left(  - \frac{1}{12 } L_{k_2} L_{k_2} L_{k_1}  + \frac{1}{6} L_{k_2} L_{k_1} L_{k_2}  - \frac{1}{12} L_{k_1} L_{k_2} L_{k_2}  \right) \\
&       + \ldots
\end{align*}
In the form of \eqref{dtwords_goldburg} this is,
$$
Z_2 = \sum_{i=1}^{\infty}{  \sum_{j}^{\binom{2-i-1}{i}}{  \hat{w}(j,i,2) \hat{g}(j,i,2) }  }.
$$
Of the three possible length two words, only $\delta t_1 \delta t_2$ has a nonzero $ \hat{g}$ coefficient, which is $ \left(\frac{1}{2} L_{k_1} L_{k_2}  - \frac{1}{2}  L_{k_2} L_{k_1} \right)$. Of the four possible length three words, only $\delta t_1^2 \delta t_2$ and $\delta t_1 \delta t_2^2$ have nonzero $ \hat{g}$; the $ \hat{g}$ for $\delta t_1^2 \delta t_2$ is $\left(  \frac{1}{12} L_{k_1} L_{k_1} L_{k_2}    - \frac{1}{6}  L_{k_1} L_{k_2} L_{k_1} + \frac{1}{12}  L_{k_2} L_{k_1} L_{k_1}  \right)$, and so on. \\
We consider the length one words in \eqref{dtwords_goldburg}. Let the sum of all the length one words in $Z_n$ be $S_n$. From \eqref{Zn_BCH} it is clear that $S_{n+1} = S_{n} + \delta t_{n+1} L_{k_{n+1}}$, and since $S_1 = Z_1 = \delta t_1 L_{k_1}$, each timestep word in $S_n$ has a coefficient $L_k \in \{L_1, L_2, \ldots , L_K \}$, i.e. from the set of all connection matrices. Thus we can write
\begin{equation}
S_n = \sum_{k=1}^{K}{t_k L_k},
\label{sndef}
\end{equation}
where $t_k$ is the sum of the $\delta t$ for every event which has used $L_k$ as the event operator, i.e., every event on face $k$. Then, $t_k$ is nothing but the face $k$'s individual time. The algorithm guarantees that at event $N$, $t_k = T$ for every face. This leads to 
\begin{equation}
Z_N = T \sum_{k=1}^{K}{ L_k} + \sum_{i=2}^{\infty}{ \sum_{j=1 }^{\binom{N+i-1}{i}}{  \hat{g}(j,i,N) \hat{w}(j,i,N) } }.
\label{S_is_good}
\end{equation}
We can write 
$$
Z_n = S_n + R_n.
$$
Comparing \eqref{S_is_good} to \eqref{lin_ex_sol} leads to 
\begin{conjecture}
$$
R_N =  \sum_{i=2}^{\infty}{ \sum_{j=1 }^{\binom{N+i-1}{i}}{  \hat{g}(j,i,N) \hat{w}(j,i,N) } } \rightarrow 0 \text{ as } N \rightarrow \infty
$$
as $N \rightarrow \infty$.
\label{conv_hope}
\end{conjecture}
To prove Conjecture \ref{conv_hope} we might invoke Remark \ref{as_dt_N} and assume that a length $i$ timestep word is $O(N^{-d_3i})$. Then as $N \rightarrow \infty$, 
$$
R_N \rightarrow  \sum_{i=2}^{\infty}{ \sum_{j=1 }^{\binom{N+i-1}{i}}{  \hat{g}(j,i,N) O(N^{-d_3i}) } } = \sum_{i=2}^{\infty}{ C(i,N) O(N^{-d_3i})  },
$$
where $C(i,N)$ is some bound on the sum of the $\hat{g}(j,i,N)$ for a given $i$. Proving a desirable bound $C(i,N)$ would require two results. First, we must ensure that no $\hat{g}(j,i,N)$ becomes unboundedly large (in some norm). Second, we must ensure that the number of nonzero $\hat{g}(j,i,N)$ for a given $i$ is sufficiently bounded. \\
Concerning the first required result, the $\hat{g}(j,i,N)$ are linear combinations of operator words and these words have the potential to become arbitrarily long as $N \rightarrow \infty$. This may not be pathological if we consider the actions of connection matrices on each other. Consider the product $L_k L_{k'}$. Unless $L_k$ and $L_k$ are associated faces, the product is an empty matrix. Typically the size of a set of associated faces is much smaller than the size of the set $\mathcal{L}$ of all connection matrices. Thus as the length of an operator word becomes arbitrarily large, the chance of it including a null pairing of connection matrices like this may become extremely high or certain. \\
For the second result, we can immediately place an upper bound on the number of $\hat{g}(j,i,N)$ as $\binom{N+i-1}{i}$, which is $O(N^i)$ as $N \rightarrow \infty$. Assuming we have the first required result, we would then have
\begin{equation}
R_N =  \sum_{i=2}^{\infty}{ O(N^{(1-d_3)i})  },
\label{RforEAS}
\end{equation}
which proves Conjecture \ref{conv_hope} if $d_3 >1$. With \asref{as_N_dm}, this becomes  $O(\Delta M ^{-2d_2(1-d_3)})$ and we have a very rough convergence result that does not consider the ordering of events, or the initial condition. This outlines how a convergence result for EAS may be formulated by taking advantage of the ability to write that scheme as a product of exponentials. A complete result will have to take into account the initial condition and how this affects event ordering, how the scheme handles event ordering in general, and unique properties of the connection matrices. For further discussion see \cite{mythesis}; also note that some additional analysis of BAS is available in \cite{Async1, mythesis}.  \\

\section{Conclusions}
\label{ext_conc}
A new type of face-based, positivity preserving asynchronous numerical method EAS has been presented, which through numerical experiment is observed to converge to a reference solution as a control parameter $\Delta M$ is decreased. We have introduced a way of expressing the new scheme as the repeated action of the matrix exponentials on an initial state vector, and used the formulation to prove the positivity preserving property of the scheme and to outline the framework for a convergence analysis.

\bibliographystyle{unsrt}
\bibliography{Bibliography_initials}

\begin{thebibliography}{10}

\bibitem{OK_plasma}
H.~Omelchenko and H.~Karimabadi.
\newblock Event-driven, hybrid particle-in-cell simulation: A new paradigm for
  multi-scale plasma modeling.
\newblock {\em Journal of Computational Physics}, 216:153--178, 2006.

\bibitem{OK_flux}
H.~Omelchenko and H.~Karimabadi.
\newblock Self-adaptive time integration of flux-conservative equations with
  sources.
\newblock {\em Journal of Computational Physics}, 216:179--194, 2006.

\bibitem{async_gas_discharge}
T.~Unfer, Jean-Pierre Boeuf, F.~Rogier, and F.~Thivet.
\newblock An asynchronous scheme with local time stepping for multi-scale
  transport problems: application to gas discharges.
\newblock {\em J. Comput. Phys.}, 227(2):898--918, 2007.

\bibitem{Async1}
D.~Stone, S.~Geiger, and G.~Lord.
\newblock Asynchronous discrete event schemes for pdes.
\newblock {\em https://arxiv.org/abs/1610.05051}.

\bibitem{mythesis}
D.~Stone.
\newblock {\em Asynchronous and exponential based numerical schemes for porous
  media flow}.
\newblock PhD thesis, Heriot-Watt, 2015.

\bibitem{roe1981approximate}
P.~L. Roe.
\newblock Approximate riemann solvers, parameter vectors, and difference
  schemes.
\newblock {\em Journal of computational physics}, 43(2):357--372, 1981.

\bibitem{toro2009riemann}
E.~F. Toro.
\newblock {\em Riemann solvers and numerical methods for fluid dynamics: a
  practical introduction}.
\newblock Springer Science \& Business Media, 2009.

\bibitem{NW}
J.~Niesen and W.~Wright.
\newblock A {K}rylov subspace algorithm for evaluating the \(\phi\)-functions
  in exponential integrators.
\newblock {\em arXiv:0907.4631v1}, 2009.

\bibitem{mrst_prime}
K.~A. Lie, S.~Krogstad, I.~S. Ligaarden, H.~M. Natvig, J. R.and~Nilsen, and
  B.~Skaflestad.
\newblock Open-source matlab implementation of consistent discretisations on
  complex grids.
\newblock {\em Computational Geosciences}, 16(2):297--322, 2012.

\bibitem{Overview}
M.~Hochbruck and A~Osterman.
\newblock Exponential integrators.
\newblock {\em Acta Numerica}, pages 209--286, 2010.

\bibitem{hochbruck-rkei-2005}
M.~Hochbruck and A.~Ostermann.
\newblock Exponential {R}unge-{K}utta methods for parabolic problems.
\newblock {\em Appl. Numer. Math.}, 53(2-4):323--339, 2005.

\bibitem{cm}
S.~M. Cox and P.~C. Matthews.
\newblock Exponential time differencing for stiff systems.
\newblock {\em Journal of Computational Physics}, 176:430--455, 2002.

\bibitem{tambue2010exponential}
A.~Tambue, G.~J. Lord, and S.~Geiger.
\newblock An exponential integrator for advection-dominated reactive transport
  in heterogeneous porous media.
\newblock {\em Journal of Computational Physics}, 229(10):3957--3969, 2010.

\bibitem{tambue2013efficient}
A.~Tambue, I.~Berre, and J.~M. Nordbotten.
\newblock Efficient simulation of geothermal processes in heterogeneous porous
  media based on the exponential rosenbrock--euler and rosenbrock-type methods.
\newblock {\em Advances in Water Resources}, 53:250--262, 2013.

\bibitem{lord2014introduction}
G.~J. Lord, C.~E. Powell, and T.~Shardlow.
\newblock {\em An Introduction to Computational Stochastic PDEs}.
\newblock Number~50. Cambridge University Press, 2014.

\bibitem{HS_lie_book}
M.~Hausner and J.~T. Schwartz.
\newblock {\em Lie groups; {L}ie algebras}.
\newblock Gordon and Breach Science Publishers, New York-London-Paris, 1968.

\bibitem{goldberg_BCH}
K.~Goldberg.
\newblock The formal power series for {${\rm log}\,e^xe^y$}.
\newblock {\em Duke Math. J.}, 23:13--21, 1956.

\bibitem{conv_goldberg}
R.~C. Thompson.
\newblock Convergence proof for {G}oldberg's exponential series.
\newblock {\em Linear Algebra Appl.}, 121:3--7, 1989.
\newblock Linear algebra and applications (Valencia, 1987).

\bibitem{list_goldberg}
M.~Newman and R.~C. Thompson.
\newblock Numerical values of {G}oldberg's coefficients in the series for
  {${\rm log}(e^xe^y)$}.
\newblock {\em Math. Comp.}, 48(177):265--271, 1987.

\bibitem{yoshida_symp}
H.~Yoshida.
\newblock Construction of higher order symplectic integrators.
\newblock {\em Phys. Lett. A}, 150(5-7):262--268, 1990.

\end{thebibliography}

\end{document}